\documentclass[reqno,11pt]{amsart}
\usepackage{amscd,amsfonts,mathrsfs,amsthm,enumerate}
\usepackage{amssymb, amsmath,bigints}
\usepackage{stmaryrd}
\usepackage{epsfig}
\usepackage[sorted-cites,nobysame]{amsrefs}
\usepackage{latexsym}
\usepackage[usenames,dvipsnames]{color}
\usepackage[T1]{fontenc}
\usepackage{calligra}
\DeclareFontShape{T1}{calligra}{m}{n}{<->s*[2.3]callig15}{}
\DeclareMathAlphabet{\mathcalligra}{T1}{calligra}{m}{n}
\usepackage[colorlinks=true,linkcolor=blue,citecolor=BrickRed,urlcolor=RoyalBlue]{hyperref}
\usepackage[margin=1in]{geometry}
\usepackage{accents}
\usepackage{upgreek}
\usepackage{csquotes}

\let\oldtocsection=\tocsection
\let\oldtocsubsection=\tocsubsection
\let\oldtocsubsubsection=\tocsubsubsection
\renewcommand{\tocsection}[2]{\hspace{0em}\oldtocsection{#1}{#2}}
\renewcommand{\tocsubsection}[2]{\hspace{1em}\oldtocsubsection{#1}{#2}}
\renewcommand{\tocsubsubsection}[2]{\hspace{2em}\oldtocsubsubsection{#1}{#2}}

\newcommand{\eps}{\varepsilon}
\newcommand{\R}{\mathbb{R}}
\newcommand{\Sph}{\mathbb{S}}
\newcommand{\n}{\mathbb{N}}
\newcommand{\h}{\mathbb{H}}
\newcommand{\disp}{\displaystyle}

\newcommand{\di}{\mathrm{d}}

\newcommand{\Hvec}{\mathbf{H}}

\newcommand{\Mbar}{N}
\newcommand{\II}{\mathrm {I\!I} }

\newcommand{\diver}{\mathrm{div}}

\def\comment#1 {{\color{red}(Comment: #1) }}

\def\dist      {\operatorname{dist}}

\def\natural  #1{{\mathbb N^{#1}}}

\def\g {\operatorname{g}}

\newtheorem{theorem}{Theorem}[section]
\newtheorem{mythm}{Theorem}

\newtheorem{lemma}[theorem]{Lemma}

\newtheorem{definition}[theorem]{Definition}
\theoremstyle{definition}
\newtheorem{remark}[theorem]{Remark}
\newtheorem{example}[theorem]{Example}

\def\pproof#1{\@ifnextchar[\opargproof
{\opargproof[\it Proof of #1.]}}
\def\opargproof[#1]{\par\noindent {\bf #1 }}

\makeatother

\numberwithin{equation}{section}

\usepackage{empheq}

\newcommand{\hmmaisum}{{\h^{m+1}}}

\newcommand{\QQ}{\mathcal{Q}}

\usepackage[colorinlistoftodos]{todonotes}

\usepackage{comment}

\usepackage{xcolor}

\newcommand{\Smmenosum}{{\mathbb{S}^{n-1}}}

\newcommand{\gI}{{\operatorname{g}_{\rm I}}}

\newcommand{\eint}{e^{\int^{a}_ z r^{-2}\di r}}

\newcommand{\bi}{\partial'_{\infty}}

\def\real     #1{{\mathbb R^{#1}}}

\newcommand{\Il}{\operatorname{I}}
	
\newcommand{\SM}{\mathcal{S}}

\newcommand{\secI}{{\sec_{\gI}}}

\begin{document}

\title[Conformal solitons]{Conformal solitons for the mean curvature flow in
hyperbolic space}

\author[L. Mari ]{\textsc{L. Mari}}
\author[J.D. Rocha de Oliveira]{\textsc{J. Rocha de Oliveira}}
\author[A. Savas-Halilaj]{\textsc{A. Savas-Halilaj}}
\author[R. Sodr\'e de Sena]{\textsc{R. Sodr\'e de Sena}}

\address{Luciano Mari \newline
Dipartimento di Matematica "Giuseppe Peano",
Universit\'a degli Studi di Torino,
10123 Torino, Italy\newline
{\sl e-mail:} {\bf luciano.mari@unito.it}
}

\address{Jose Danuso Rocha de Oliveira\newline
Universidade Federal do Cear\'a,
Departamento de Matem\'atica,
60440900 Fortaleza, Brazil\newline
{\sl e-mail:} {\bf danusorocha@gmail.com}
}

\address{Andreas Savas-Halilaj\newline
{Department of Mathematics,
Section of Algebra \!\&\! Geometry, \!\!
University of Ioannina,
45110 Ioannina, Greece} \newline
{\sl E-mail:} {\bf ansavas@uoi.gr}
}

\address{Renivaldo Sodr\'e de Sena \newline
IFCE,  Departamento de Ensino,
62960-000 Campus Tabuleiro do Norte, Brazil\newline
{\sl e-mail:} {\bf renivaldo.sena@ifce.edu.br}
}

\date{}

\begin{abstract}
In this paper we study conformal solitons for the mean curvature flow in hyperbolic space $\h^{n+1}$. Working in the upper half-space model, we focus on horo-expanders, which relate to the conformal field $-\partial_0$. We classify cylindrical and rotationally symmetric examples, finding appropriate analogues of grim-reaper cylinders, bowl and winglike solitons. Moreover, we address the Plateau and the Dirichlet problems at infinity. For the latter, we provide the sharp boundary convexity condition to guarantee its solvability, and address the case of noncompact boundaries contained between two parallel hyperplanes of $\partial_\infty \h^{n+1}$. We conclude by proving rigidity results for bowl and grim-reaper cylinders.  
\end{abstract}

\maketitle
\setcounter{tocdepth}{1}
\section{Introduction}
Let $M$ and $N$ be complete connected Riemannian manifolds and $f: M \to \Mbar$
an isometric immersion. A {\em mean curvature flow} (MCF for short)
issuing from $f$ is a smooth map $F:M\times [0,T)\to \Mbar$ such that each $F_t:M\to\Mbar$,
$F_t(\cdot)=F(\cdot\,,t),$
is an immersion and satisfies the evolution equation
\begin{equation}\label{MCF}
\left\{ \begin{array}{l}
\dfrac{\partial F_t}{\partial t}= {\bf H}(F_t), \\[0.3cm]
F_0 = f,
\end{array}\right.
\end{equation}
where ${\bf H}(F_t)$ is the unnormalized mean curvature vector field of $F_t$.
If $M$ is compact, then \eqref{MCF} admits a smooth and unique (up to reparametrizations) solution, 
see for example \cites{baker,huiskenpolden,smoczyk}. Although most of the literature regards the MCF in 
Euclidean space, there are significant applications of the MCF in other ambient manifolds,
see \cite{smoczyk} for an overview.

In this paper we will consider a special class of solutions to the MCF.
Let $X$ be a smooth vector field on a Riemannian manifold $N$ and
$\varPhi: \mathscr{D} \subset N\times \R \to N$ its associated flow with maximal domain $\mathscr{D}$. A solution $F:M\times (0,T)\to\Mbar$ to the MCF is said to \emph{move along $X$} if there 
exists an immersion $f:M\to \Mbar$, a reparametrization $s:(0,T)\to \R$
of the flow lines of $X$ and a 1-parameter family of diffeomorphisms $\eta:M\times (0,T)\to M$
such that
\begin{equation}\label{self_similar}
F(x,t) = \varPhi\big(f(\eta(x,t)),s(t)\big),\qquad (x,t)\in M\times (0,T).
\end{equation}
While the definition is meaningful for arbitrary $X$, in this paper we focus on conformal vector fields since in this case the MCF ``preserves'' the shape of the evolved submanifold. A MCF moving along a conformal field $X$ is said to be a \emph{self-similar solution with respect to $X$.}
Self-similar solutions have played an important role in the development of the theory
of the MCF, as they serve as comparison solutions to
investigate the formation of singularities. Differentiating \eqref{self_similar} with respect to $t$ and estimating at $0$, we obtain the equation
$$ 
\Hvec=s'(0)X^{\perp},
$$
where $\{\cdot\}^{\perp}$ is the orthogonal projection on the normal bundle of $f$.
This motivates the following definition:

\begin{definition}
An isometric immersion $f:M\to\Mbar$ satisfying the elliptic equation
\begin{equation}\label{MCFS}
\Hvec= cX^{\perp},
\end{equation}
where $c\in\R$, is called a {\em soliton with respect to} $X$ with {\em soliton constant} $c$. If $X$ is a gradient field {\rm (}resp. a conformal field or a parallel field\,{\rm )}, then $f$ is named a gradient
{\rm (}resp. a conformal or a translating\,{\rm)} soliton.  
\end{definition}

\begin{remark}
Let us make some comments regarding solitons in general ambient spaces:
\begin{enumerate}[\rm(1)]
\item
The case where $X$ is conformal and closed was considered in \cites{arezzo,alias,CMR,smoczyk1}.
A soliton with respect to $X$ with constant $c$ is a soliton with respect to $cX$ with constant $1$. 
However, in what follows, it will be more convenient to specify $X$ from the very beginning and keep $c$ 
as a parameter.
\smallskip
\item
The fact that $f$ solves \eqref{MCFS} may not imply that the MCF issuing from $f$ moves along $X$. This is the case, however, when $X$ is a Killing field, since $X$ generates a $1$-parameter group of isometries, \cite{hunger1,hunger2}. Nevertheless, the study of (possibly non-Killing) gradient solitons in more general ambient spaces can also be justified from the parabolic point of view, as shown by Yamamoto \cite{yamamoto}. The starting point is the investigation of the MCF in a Ricci flow background; see \cite{Lott,MMT}. More precisely, let $\{\g_t\}$
a smooth $1$-parameter family of Riemannian metrics on $N$ and $F_t:M\to N$ a smooth $1$-parameter
family of immersions for $t\in[0,T)$. We say that $\{(\g_t,F_t)\}$ is a solution to the {\em 
Ricci-mean curvature flow} if the following system is satisfied  
\begin{equation}\label{RMCF}
\left\{ \begin{array}{lll}
\dfrac{\partial\hspace{0.6pt} {\g}_t}{\partial t}&=&-2\operatorname{Ric}(\g_t),\\[0.3cm]
\dfrac{\partial F_t}{\partial t}&=&\quad {\bf H}(F_t),
\end{array}\right.
\end{equation}
where ${\bf H}(F_t)$ denotes the mean curvature vector field of $F_t : M \to (N,\g_t)$. One interesting case is that of a MCF in a gradient shrinking Ricci soliton $(N,\g,u)$. In this situation, $\g$ and $u$ satisfy
$$
\operatorname{Ric}(\g)+ \operatorname{Hess}_{\g}(u) -(1/2)\g=0
$$
and thus $\g_t=(T-t)\varPhi^*_t\g$, where $\varPhi_t:N\to N$, $t\in(0,T)$ is the flow of the vector field
$$V=\frac{\nabla u}{T-t}.$$
Yamamoto \cite{yamamoto} proved the following results:
\smallskip
\begin{itemize}
\item
Let $(N,\g,u)$ be a shrinking gradient Ricci soliton and let  $f:M\to N$ be a soliton satisfying
$${\bf H}(f)=-(\nabla u)^{\perp}.$$
Then, $F:M\times[0,T)\to N$ given by $F=\varPhi^{-1}_t\circ f$ forms,  
up to tangential reparametrizations, a solution to the Ricci-mean curvature flow \eqref{RMCF};
see \cite[Proposition 4.3]{yamamoto}.
\medskip
\item
Let $(N,\g,u)$ be a compact shrinking gradient Ricci soliton, $M$ a compact manifold and let
$F:M\times[0,T)\to N$, $T<\infty$, be a
Ricci-mean curvature flow defined by \eqref{RMCF}. Suppose that
the second fundamental forms ${\bf A}(F_t)$ of $F_t$ satisfy
$$
\max_{M}|{\bf A}(F_t)|<\frac{C}{\sqrt{T-t}},
$$
where $C$ is a positive constant. For any increasing sequence $\{s_j\}$ of numbers
tending to infinity and any sequence of points $\{x_j\}$ in $M$, the family of the rescaled pointed immersions $G_{s_{j}}:(M,x_{j})\to N$ given by
$G_{s_{j}}=\varPhi_{t_{j}}\circ F_{t_{j}}$, where $t_j$ is defined by $s_{j}=-\log(T-t_{j}),$
subconverges in the Cheeger-Gromov sense to an immersion $f_{\infty}:M_{\infty}\to N$ of a complete Riemannian manifold satisfying the equation
$$
{\bf H}(f_{\infty})=-(\nabla u)^{\perp};
$$
for more details we refer to \cite[Theorem 1.5]{yamamoto}.
\end{itemize}
\end{enumerate}
\end{remark}

\begin{remark}\label{rem_minimal}
When $X$ is a gradient field, a further motivation for studying solutions to \eqref{MCFS} 
is the
tight relation to the theory of minimal submanifolds; 
see for example
\cite{hunger1,hunger2,hunger3,colding1,colding2,ilmanen,IRS,IR,smoczyk1,gromov,CZ}.
\begin{enumerate}[\rm(1)]
\item
Isometric immersions $f : M^k \to (N^{n+1},\g)$ satisfying the soliton equation
\begin{equation}\label{MCFS_grad}
{\bf H} = (\nabla u)^\perp,
\end{equation}
for $u \in C^\infty(N)$ are precisely the stationary points of the weighted volume functional
\[
\Omega \subset M \mapsto \int_\Omega e^{u(f)} \di x,
\]
where $\di x$ is the induced Riemannian measure on $M$. Consequently, solitons are particular examples of $u$-minimal submanifolds. By considering the conformally related metric
\[
\gI(k) = e^{\frac{2u}{k}}\g
\]
and endowing $M$ with the induced metric ${\rm h}(k) = f^*\gI(k)$, $e^u \di x$ is the Riemannian volume measure of ${\rm h}(k)$ and thus $f$ solves \eqref{MCFS_grad} if and only if $f : (M, {\rm h}(k)) \to (N, \gI(k))$ is minimal. The metric $\gI(k)$ is called the \emph{Ilmanen metric}.
\item
\smallskip
According to a result of Smoczyk \cite{smoczyk1}, there exists a 1-1 correspondence between
gradient conformal solitons in $N$ and minimal submanifolds in a suitable warped product constructed out of $N$. Smoczyk proved that if $f : M \to (N,\g)$ is a soliton with respect to a conformal gradient field $\nabla u$, then its associated submanifold $\bar{M}=\R\times M$ is minimal in $\bar{N}=\R\times N$ equipped with the warped metric $\bar{\g}_{(s,x)}=e^{2u(x)} \di s^2+\g_{x}$. Moreover, he proved that a submanifold $M$ in $N$ converges to
a conformal soliton under the MCF if and only if its associated submanifold $\bar{M}$ 
converges to a minimal submanifold under a rescaled MCF in $\bar{N}$.
\end{enumerate}
\end{remark}

In the present paper, we will investigate conformal solitons in the hyperbolic space 
\begin{equation}\label{eq_uhs}
\h^{n+1}=\big\{(x_0,x_1,\dots,x_n)\in \R^{n+1}:x_0>0\big\}, \quad  
\g_{\h}=x^{-2}_0{\sum}_{i=0}^{n} \di x_i^2.
\end{equation}
The class of conformal vector fields of $\h^{n+1}$ is particularly rich. Thus, we 
shall restrict ourselves to solitons with respect to $-\partial_0$, i.e., solutions to 
\begin{equation}\label{mcftr}
{\bf H}= -\partial_0^{\perp} = (\nabla x_0^{-1})^{\perp}.
\end{equation}
Such solitons correspond to ``limit self-expanders'' which we call {\em horo-expanders}. It turns out that 
horo-expanders share many similarities with translators in Euclidean space. One may suspect that the 
analogy is a trivial consequence of the fact that the model \eqref{eq_uhs} is conformal to the Euclidean 
upper half-space. However, a direct computation shows that a $k$-dimensional conformal soliton in $\h^{n+1}$ with respect to
$-\partial_0$ is 
a soliton in the Euclidean half-space $\R^+\times\R^n$ with respect to
$(- x^{-1}_0 - k x_0^{-2})\partial_0$, which is not conformal.
Hence, it seems to us that a duality between conformal solitons in $\h^{n+1}$ and
translators in $\R^{n+1}$ is hardly obtainable via simple transformations.

In view of Remark \ref{rem_minimal}(1), we can regard a $k$-dimensional soliton with respect to $-\partial_0$ as minimal submanifold of $(\h^{n+1}, \gI(k))$, where
\begin{equation}\label{Imetric}
\gI(k)=e^{\frac{2}{kx_0}}\g_{\mathbb{H}}.
\end{equation}
This allows us to relate the existence of a soliton with prescribed boundary contained in the boundary at 
infinity $\partial_\infty \h^{n+1}$ to the existence of minimal submanifolds in Riemannian manifolds.
It turns out that in codimension one, both the Plateau and the Dirichlet problem at infinity are solvable, the 
latter under the additional condition that the boundary is mean convex. However, a distinction shall be made between boundary points of $\partial_\infty \h^{n+1}$: in the upper half-space model \eqref{eq_uhs} the boundary at infinity can be represented as the union of
$$\partial_{\infty}'\h^{n+1}=\{x_0=0\}\qquad\text{and}\qquad
{\rm p}_{\infty}=\partial_{\infty}\h^{n+1}\backslash\partial_{\infty}'\h^{n+1},
$$
which behave quite differently for the Ilmanen metric.

\begin{remark}
Hereafter, $\partial'_{\infty}\h^{n+1}$ will be given the Euclidean metric $\sum_{j=1}^n \di x_j^2$, and metric quantities (balls, hyperplanes, neighbourhoods, etc..) will be considered with respect to it.
\end{remark}

Regarding Plateau's problem, we show its solvability for hypersurfaces.
The higher codimensional case 
remains as an open problem; see Section \ref{sec4}.
Before stating our result,
let us recall some standard notations from geometric measure theory: given a closet subset  $W$ with
locally finite perimeter in
a smooth manifold $N$, we denote by $[W]$ its associated rectifiable current. If $M$ is a
rectifiable current we denote by
$\operatorname{spt} M$ its support and by $\partial M$ its boundary; for more details we refer to \cite{simon}.

\begin{mythm}[Plateau's problem]\label{teoPlateau}
Let $\Sigma \subset \bi \h^{n+1}$ be the boundary of a relatively compact subset 
$A\subset \bi \h^{n+1}$ with $A=\overline{\operatorname{int}(A)}$.
Then, there exists a closed set $W$ of local finite perimeter in $\h^{n+1}$ with $\partial_{\infty}W=A$ such that
$M=\partial[W]$ is a conformal soliton for $-\partial_0$ on the complement of a closed set  $\mathrm{S}$ of Hausdorff dimension $\dim_{\mathscr{H}}(\mathrm{S})\le n-7$, and that $\partial_{\infty}\operatorname{spt}(M)=\Sigma$.
Furthermore, when $n<7$, then $M$ is a properly embedded smooth hypersurface of $\h^{n+1}$.
\end{mythm}

We next focus on hypersurfaces which are graphs of the form
$$
\Gamma(u)=\big\{(u(x);x)\in\h^{n+1}=\R^{+}\times\R^n:x\in \Omega\subset\R^{n}\big\},
$$
where $u\in C^{\infty}(\Omega)$. Given
$0 \le \phi \in C(\partial \Omega)$, it turns out that $\Gamma(u)$ is a soliton with respect to
$-\partial_0$ with boundary $\Gamma(\phi)$ if and only if $u$ satisfies
\begin{equation}\label{QQ}
\left\{ \begin{array}{lll}
\diver\left(\dfrac{D u}{\sqrt{1+|D u|^2}}\right) = -\dfrac{1+nu}{u^2\sqrt{1+|D u|^2}} &
\text{on } \,\,\, \Omega, \\[0.4cm]
u>0 &\text{on } \,\,\, \Omega,\\ [0.2cm]
u= \phi &\text{on } \, \partial\Omega,
\end{array}\right.
\end{equation}
where $D,\diver,|\cdot |$ are the gradient, divergence and norm in the Euclidean metric. In particular, for $\phi\equiv 0$, we obtain a complete graphical soliton whose boundary
at infinity is $\partial\Omega$. The right-hand side of \eqref{QQ} becomes undefined as $u=0$,
which calls for some care. However, we will see that the terms $u$ and $|D u|$ contribute in opposite directions to
the size of $u$. We show the following result, which parallels the seminal one by Jenkins \& Serrin \cite{jenkins}
for minimal graphs:

\begin{mythm}[Dirichlet's problem]\label{teoDir}
Let $\Omega \subset \bi \h^{n+1}$ be an open, connected subset with $C^3$-smooth boundary $\partial\Omega$. 
Assume that $\Omega$ is contained between two parallel hyperplanes of $\bi \h^{n+1}$, and denote by 
$H_{\partial \Omega}$ the Euclidean mean curvature of $\partial \Omega$ in the direction pointing towards
$\Omega$.  
\begin{enumerate}[\rm(1)]
\item If $H_{\partial \Omega} \ge 0$ on $\partial \Omega$, then for each continuous bounded function $\phi : \partial \Omega \to [0,\infty)$ there exists a function $u : \overline{\Omega} \to [0,\infty)$ such that:
\smallskip
\begin{enumerate}
\item[\rm(a)]
$u>0$ on $\Omega$ and the graph $\Gamma(u) \subset \h^{n+1}$ is a conformal soliton with respect to
$-\partial_0$.
\medskip
\item[\rm(b)]
$u\in C^{\infty}(\Omega)\cap C(\overline{\Omega})\cap L^\infty(\Omega)$ and $u\equiv \phi$ on $\partial \Omega$.
\medskip
\end{enumerate}
\item If $\Omega$ is bounded and $H_{\partial \Omega}(y) < 0$ for some $y \in \partial \Omega$, then there exists a continuous boundary value function
$\phi : \partial \Omega \to (0,\infty)$ such that no
$u : \overline{\Omega} \to [0,\infty)$  satisfying the properties in $\rm(1)$ does exist. 
\end{enumerate}
\end{mythm}

\begin{remark}\label{rem_diri_intro}
Let us make some comments about the conclusions of Theorem \ref{teoDir}.
\begin{enumerate}[\rm(1)]
\item If $\Omega$ is bounded, then the function $u$ realizing (a), (b) is unique, by a direct application of the comparison theorem. It would be interesting to investigate the uniqueness problem for noncompact $\Omega$. 
\item
Similar (non-degenerate) Dirichlet problems were considered for prescribed mean curvature graphs in warped 
product manifolds; see for example 
\cite{anderson2,lin,guan,jenkins,serrin2,dajczer2,dajczer3,bonorino,casteras2,casteras3}. Among them, the only 
applicable result to \eqref{QQ} is \cite[Theorem 1.1]{casteras3}, which guarantees the solvability of
\begin{equation}\label{eq_dir_chh}
\left\{\begin{array}{ll}
\disp \diver \left( \frac{D u}{\sqrt{1+|D u|^2}} \right) = \frac{f(u)}{\sqrt{1+|D u|^2}} & \quad \text{on a smooth } \, \Omega \subset \R^n, \\[0.4cm]
u = \phi & \quad \text{on } \, \partial \Omega,
\end{array}\right.
\end{equation}
for $C^{2,\alpha}$-smooth positive $\phi$ provided that $f \in C^1(\R)$ and $H_{\partial \Omega}$ satisfy:
\begin{equation}\label{eq_curv_F}
(n-1) \kappa \doteq {\sup}_{\R} |f| < \infty \qquad\text{and}\qquad H_{\partial \Omega} \ge (n-1)\kappa.
\end{equation}
However, application to \eqref{QQ} would force a lower bound on $H_{\partial \Omega}$ that diverges as $\min_{\partial \Omega} \phi \to 0$.

\smallskip
\item Serrin discussed in
\cite[Chapter IV, pages 477-478]{serrin2} the solvability and the non-solvability
of the Dirichlet problem for equations of the form
\begin{equation}\label{S-B1}
\diver\left(\frac{Du}{\sqrt{1+|Du|^2}}\right)=\frac{\Lambda}{(1+|Du|^2)^{\theta}}
\quad\text{and}\quad
\diver\left(\frac{Du}{\sqrt{1+|Du|^2}}\right)=\frac{C u}{(1+|Du|^2)^{\theta}},
\end{equation}
where $\Lambda$, $C$  and $\theta$ are constants and $C>0$. Note that,
for $\Lambda=1$ and $\theta=1/2$, the equation \eqref{S-B1} describes
a translating graphical soliton of the MCF in the Euclidean space. Both equations \eqref{S-B1} appear in an old paper of
Bernstein \cite{bernstein}. As a matter of fact, Bernstein studied the
2-dimensional case and showed that the corresponding Dirichlet problems are solvable
for arbitrary analytic boundary data in an arbitrary strictly convex analytic domain only
for special values of $\theta$.
The problem that we treat in \eqref{QQ} does not fall in the class of equations defined in 
\eqref{S-B1}.

\smallskip
\item Jenkins and Serrin \cite{jenkins} constructed data $(\partial\Omega,\phi)$
for which the Dirichlet problem for the minimal surface equation in the Euclidean space is not solvable.
In their work, $\partial\Omega$ has negative mean curvature at a given point and the oscillation of $\phi$ can be made arbitrarily small. However, this is not the case for the boundary data we provide in Theorem \ref{teoDir}(2), whose oscillation shall be at 
least a fixed amount. From the proof of
Theorem \ref{teoDir}(2), we may suspect the impossibility to produce boundary data with arbitrarily small oscillation for which \eqref{QQ} is not
solvable.
\smallskip
\item
As a direct application of the geometric maximum principle we can show that there are no solutions to
$$
\diver\left(\dfrac{D u}{\sqrt{1+|D u|^2}}\right) + \dfrac{1+nu}{u^2\sqrt{1+|D u|^2}}=0 \qquad \text{on the entire $\R^n$}.
$$
\end{enumerate}
\end{remark}

Another goal of our paper is to construct and classify complete codimension one solitons
with symmetries.
Despite their own interest, such solitons serve as important barriers and will be used in the 
proofs of Theorems A and B.
Up to rotation, all examples are generated by special curves $\gamma:I\to\R^+\times\R$,
$$\gamma(t)=(x_0(t),x_1(t)),$$ 
in the $x_0x_1$-plane. The first family to be considered is that of cylindrical solitons:
\begin{equation}\label{eq_cylindrical}
M = \big\{ (x_0, x_1,\ldots, x_n) \in \h^{n+1} \ : \ (x_0,x_1) \in \gamma(I)\big\}.
\end{equation}

We shall prove that the only such examples are the {\em grim-reaper cylinders}, described by the following:

\begin{mythm}[Grim-reaper cylinders]\label{teoexiGR}
If $M$ is a complete soliton for $-\partial_0$ of the type \eqref{eq_cylindrical}, then it has the following properties:  
\begin{enumerate}[\rm(1)]
\item $\partial_\infty M$ is a pair of parallel hyperplanes $\pi_1 \cup \pi_2$ in $\partial'\h^{n+1}$.
\smallskip
\item $M$ is contained between the two totally geodesic hyperplanes $\Pi_1$ and $\Pi_2$ of $\h^{n+1}$ with $\partial'_\infty \Pi_1 = \pi_1$
and $\partial'_\infty \Pi_2 = \pi_2$.
\smallskip
\item $M$ is symmetric with respect to the reflection sending $\Pi_1$ to $\Pi_2$, and invariant with respect to translations fixing $\Pi_1$ and $\Pi_2$.
\smallskip
\item $M$ is $($Euclidean$)$ convex with respect to the direction $-\partial_0$.
\end{enumerate}
We name $M$ a grim-reaper cylinder. Denote with $h = \max x_0(M)$, and let $\mathscr{G}^h$ be the grim-reaper cylinder isometric to $M$ which is symmetric with respect to the hyperplane $\{x_1=0\}$. Then, the family $\{\mathscr{G}^h\}_{h \in \R^+}$ foliates $\h^{n+1}$. In particular, given a pair of parallel hyperplanes $\pi_1,\pi_2 \subset \partial'_\infty\h^{n+1}$, the grim-reaper cylinder with $\partial_\infty M = \pi_1 \cup \pi_2$ exists and is unique.
\end{mythm}

Another way to construct complete conformal solitons is by rotating $\gamma$ around the $x_0$-axis. From this procedure we obtain the hyperbolic winglike and bowl solitons, similar to those existing in Euclidean 
space; see for more details \cite{clutterbuck} and \cite{fra14}. We here report a simplified, slightly 
informal statement of our main existence and uniqueness result. For the precise one, including further 
properties of $\gamma$, we refer the reader to Lemma \ref{lem_1_phi} and 
Theorem \ref{preD}.

\begin{mythm}[Rotationally symmetric solitons]\label{TCyl}
There are exactly two families of curves $\gamma$ giving rise to a complete, smooth rotationally 
symmetric conformal soliton with respect to $-\partial_0$. They are depicted in Figure \ref{Pic-7b}, and named
$\gamma_B$ and $\gamma_W$. We call a \emph{bowl soliton} the one obtained by rotating
$\gamma_B$, and a \emph{winglike soliton} that obtained by rotating $\gamma_W$. The 
following holds:
\begin{enumerate}[\rm(1)]
\item $\gamma_B$ is a strictly concave graph over a domain $(0,h_B)$ of the $x_0$-axis, and meets the 
$x_1$ and $x_0$ axes orthogonally.
\smallskip
\item $\gamma_W$ is a bigraph over a domain $(0,h_W)$ of the $x_0$-axis, and does not touch the 
$x_0$-axis. The upper graph $($the one from $q_1$ to $p_1$$)$ is strictly concave, while the lower
graph is the  union of a strictly convex branch with a unique minimum $($from $p_1$ to $p_2$$)$ and
a strictly concave  branch $($from $p_2$ to $q_2$$)$. Moreover, $\gamma_W$ meets the $x_1$-axis orthogonally at two distinct points {\rm (}$q_1 \neq q_2${\rm )}.
\end{enumerate}
\end{mythm}
\begin{figure}[h!]
	\centering
	\includegraphics[width=.5\textwidth]{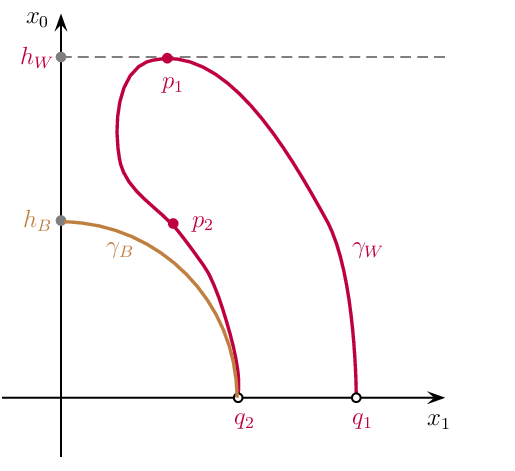}
	\caption{Curve generating the winglike soliton.}\label{Pic-7b}
\end{figure}

It turns out that bowl solitons foliate $\h^{n+1}$, see Lemma \ref{r2est}. As a direct consequence, we deduce
the following uniqueness property which may be viewed as an analogue of \cite[Theorem A]{fra14}.

\begin{mythm}[Uniqueness of the bowl soliton]\label{TBowl}
For any Euclidean ball $B_R \subset \partial'_\infty \h^{n+1}$, there exists a bowl soliton $M$ with $\partial_\infty M = \partial B_R$, and it is the unique properly immersed soliton with respect to $-\partial_0$ with boundary at infinity $\partial B_R$. 
\end{mythm}

Regarding the uniqueness of grim-reaper cylinders, the problem is more subtle. Quite differently from the 
Euclidean case, hyperbolic grim-reaper cylinders foliate the hyperbolic space, which is quite helpful. 
Nevertheless, getting  uniqueness under the only assumption that $\partial_\infty M$ is a pair of parallel 
hyperplanes of $\partial_\infty'\h^{n+1}$ seems difficult. Our last result is that $M$ is a grim-reaper 
cylinder provided that the height $x_0$ is bounded and that $M$ is graphical in a small region
$\{x_0< \tau\}$, see Definition \ref{GR}. A similar result was proved in \cite{fra15,gamamartin} for MCF 
translators in Euclidean space under the stronger assumption that they are $C^1$-asymptotic outside a 
cylinder to a grim-reaper cylinder. In our setting, we will use a calibration argument to get rid of the 
$C^1$-bound. 
 
\begin{definition}\label{GR}
A properly embedded hypersurface $M\subset\h^{n+1}$ is said to satisfy the GR-property
{\rm (}see Figure \ref{Pic-12}\,{\rm)} if the following
conditions are satisfied:
\begin{enumerate}[\rm(1)]
\item
$\bi M=\pi_1\cup\pi_2$, where $\pi_1$ and  $\pi_2$ are parallel hyperplanes of $\partial'_{\infty}\h^{n+1}$.
\medskip
\item
The $x_0$-component of $M$ is bounded.
\medskip
\item
There exist $\tau>0$, a pair of $($Euclidean$)$ hyperplanes $\mathcal{H}_j\subset\h^{n+1}$ and a pair
of functions $\varphi_j:\mathcal{H}^{\tau}_j=\mathcal{H}_j\cap\{x_0<\tau\}\to\R$,
$j\in\{1,2\}$, such that:
\medskip
\begin{enumerate}[\rm(a)]
\item
$\partial'_{\infty} \mathcal{H}_j=\pi_j$.
\medskip
\item
The wings
$
\mathcal{W}_j=\{x+\varphi_j(x)\nu_j:x\in \mathcal{H}_j^{\tau}\}
$
are contained in $M$, where $\nu_j$ is a fixed $($Euclidean$)$ unit normal to $H_j$.
\smallskip
\item
$M\cap\{x_0<\tau\}$ is the portion of $\mathcal{W}_1\cup\mathcal{W}_2$ inside $\{x_0<\tau\}$.
\end{enumerate}
\end{enumerate}
\end{definition}

\begin{figure}[h!]
	\centering
	\includegraphics[width=.72\textwidth]{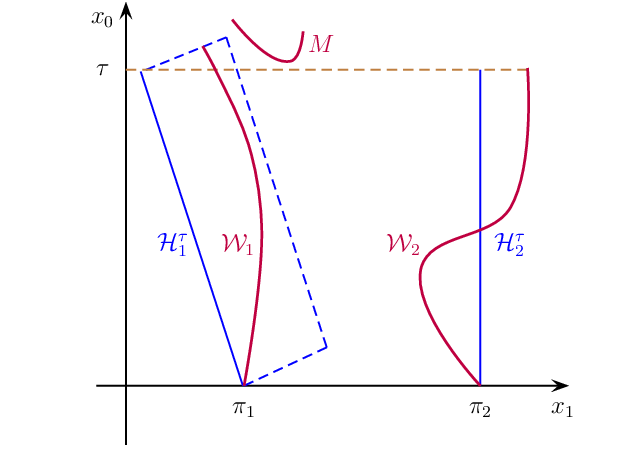}
	\caption{GR-property.}\label{Pic-12}
\end{figure}

Observe that condition (a) in Definition \ref{GR} implies that
$$
\forall \, y \in \pi_j, \qquad \lim_{x\to y}\varphi_j(x)= 0.
$$
However, a priori the limit may not be uniform in $y$, an assumption which was required in \cite{fra15,gamamartin}.

\begin{mythm}[Uniqueness of the hyperbolic grim-reaper cylinder]\label{teoGR}
Let $M\subset\h^{n+1}$ be a properly embedded conformal soliton with respect
to $-\partial_0$ satisfying the GR-property. Then $M$ coincides with a grim-reaper cylinder.
\end{mythm}

The structure of the paper is as follows. Section \ref{sec2} we set up the notation and derive basic
properties for conformal solitons in the hyperbolic space.
In Section \ref{sec3}, we examine symmetric horo-expanders and prove Theorems \ref{teoexiGR}, \ref{TCyl} and 
\ref{TBowl}. Section \ref{sec4} is devoted to
the Plateau problem at infinity, while in Section \ref{secdir} we consider the
Dirichlet problem at infinity and prove Theorem \ref{teoDir}.
Finally, in Section \ref{sec5}, we prove Theorem \ref{teoGR}.

\section{Preliminaries}\label{sec2}

In this section we fix the notation and review some basic formulas.

\subsection{Generalities about soliton solutions}
A vector field $X$ in an $n$-dimensional Riemannian manifold $(N,\g)$ is called {\em conformal} if the Lie derivative of $\g$ in direction $X$ satisfies 
$$
\mathcal{L}_X\g= 2 \psi \g \qquad \text{for some } \, \psi \in C^\infty(N).
$$
Taking traces, $\psi = \diver_{\g} X/n$. By relating the Lie derivative to the Levi-Civita connection $\nabla$, a conformal vector field is characterized by the identity
$$
\g(\nabla_YX,Z)+\g(\nabla_ZX,Y)=\frac{2}{n}\big(\diver_{\g}{X}\big)\g(Y,Z)
$$
for $Y,Z\in\mathfrak{X}(N)$. If the field $X$ is conformal and divergence free, it is called {\em Killing}.
If the vector field $X$ satisfies
$$
\nabla_YX=\frac{1}{n}\big(\diver_{\g}{X}\big)Y,
$$
namely, the dual form $X_\flat$ is closed, then $X$ is called {\em closed conformal}. It is a well-known fact that a Riemannian manifold possessing a non-trivial closed conformal vector field is locally isometric to a warped product with a $1$-dimensional factor; see for instance \cite[page 721]{montiel}.
The hyperbolic space possesses various conformal vector fields $X$ and the study of the
corresponding solitons is interesting. Let us see some explicit examples here:

\begin{example}
Denote with $\g_{\mathbb{S}},$ $\g_{\mathbb{R}}$, $\g_{\h}$ the Riemannian metrics of
the $n$-dimensional unit sphere $\Sph^n$, the Euclidean space $\R^n$ and of the hyperbolic space $\h^n$.
\begin{enumerate}[\rm(1)]
\item Consider for the hyperbolic space the model $\h^{n+1}\backslash \{0\} = \R^+ \times \Sph^n$ equipped with the
metric $\di r^2 + \sinh^2 (r)\g_{\mathbb{S}}$. Then the vector field $X = \sinh(r) \partial_r$ is conformal. We call solitons for $cX$
{\em expanders} if $c>0$, and {\em shrinkers} if $c<0$. Also, rotation vector fields in the $\Sph^n$ factor extend to Killing fields and give rise to solitons that we call {\em rotators}.

\smallskip
\item Consider for the hyperbolic space the model $\h^{n+1} = \R \times \R^n$ endowed with the warped
metric $\di r^2 + e^{2r} \g_{\mathbb{R}}$. Then, the vector field $X = e^r \partial_r$ is conformal. The 
change of coordinates $x_0 = e^{-r}$ gives rise to an isometry with the upper half-space model sending 
$X$ to the field $-\partial_0$. Solitons with respect to $cX$ are called {\em horo-expanders} if $c>0$ and 
{\em horo-shrinkers} if $c<0$. 
\smallskip
\item Consider for the hyperbolic space the model $\h^{n+1} = \R \times \h^n$ equipped with the Riemannian
metric
$\di r^2 + \cosh^2(r) \g_{\h}$. Then $X = \cosh (r)\partial_r$ is conformal. Given that
$\cosh(r)$ is  even, we restrict to $c>0$. A soliton with respect to $cX$ {\em shrinks} where $r<0$ and 
{\em expands} where $r>0$.
\smallskip
\item Consider for the hyperbolic space the model $\h^{n+1} = \R \times \h^n$ with the metric
$\cosh^2 (\rho) \di s^2 + \g_{\h}$, with $\rho$ is the distance in $\h^n$ to a fixed point. Then the
vector field $X = \partial_s$ is Killing. In the upper half-space model, it corresponds to the position vector field
$$
X = \sum_{j=0}^n x_j \partial_j.
$$
\item Consider for the hyperbolic space the model $\h^{n+1} = \R \times \h^n$ with the metric
$e^{2\rho} \di s^2 + \g_{\h}$, with $\rho$ a Busemann function in $\h^n$. The vector field $X = \partial_s$ is 
Killing. In the upper half-space model, it corresponds to the vector field $\partial_j$,  for 
$j \in \{1,\ldots, n\}$. 
\end{enumerate}
\end{example}

\subsection{Conformal solitons and the Ilmanen metric}
Suppose that $M$ and $N$ are connected manifolds with dimensions $k$ and $n+1$, respectively,
with $k\le n$. Let ${\rm h}_1$ and ${\rm h}_2$ be metrics on $N$ which are conformally related:  
$$
{\rm h}_2=\lambda^2 {\rm h}_1 \qquad \text{for some } \, 0<\lambda \in C^\infty(N).
$$
Assume that $f:M\to N$ is an immersion and denote by
$\g_j=f^*{\rm h}_j$ the corresponding induced metrics. Then, the second fundamental forms ${\bf A}_j$ of
$f_j=f:(M,\g_j)\to (N,{\rm h}_j)$ are related by
\begin{equation}\label{secformconf}
{\bf A}_{2}(X,Y)={\bf A}_{1}(X,Y)-\g_1(X,Y)\big(\nabla^{1}\log\lambda\big)^{\perp},
\end{equation}
for any $X,Y\in\mathfrak{X}(M)$. Here, $\nabla^{1}$ stands for the Levi-Civita connection of ${\rm h}_1$ and
$\{\cdot\}^{\perp}$ denotes the orthogonal projection with respect to
${\rm h}_1$ on the normal bundle of $f_1$. Taking traces, we see that the corresponding mean curvature vectors ${\bf H}_1$ and ${\bf H}_2$ are related
by
\begin{equation}\label{meanconf1}
{\bf H}_2 = \lambda^{-2}\big\{{\bf H}_1- k(\nabla^{1}\log \lambda)^{\perp} \big\}. 
\end{equation}
As immediate consequence of the above formulas we obtain the following:
\begin{lemma}
Let $M\subset(\h^{n+1},\g_{\h})$ be a $k$-dimensional soliton of the MCF with respect to $-\partial_0$. Then, the following hold:
\begin{enumerate}[\rm(1)]
\item
$M$ is a minimal submanifold of the Ilmanen space $(\h^{n+1},\gI(k))$ given
in \eqref{Imetric}.

\smallskip
\item
$M$ is a soliton of the Euclidean half-space $\R^+\times\R^n$ with respect to the
field $(- x^{-1}_0 - kx_0^{-2}) \partial_0$, which is not
conformal.
\end{enumerate}
\end{lemma}

\section{Symmetric conformal solitons}\label{sec3}
In this section we examine special conformal solitons
and will prove Theorems \ref{teoexiGR}, \ref{TCyl} and \ref{TBowl}.

\subsection{The convex hull property} Recall that a $k$-dimensional conformal soliton in 
$\h^{n+1}$ with respect to $-\partial_0$ can be regarded as 
a minimal  submanifold when $\h^{n+1}$ is equipped with the Ilmanen metric 
$$
\gI(k)=e^{\frac{2}{k x_0}}\g_{\h}.
$$
Hence, solitons are real analytic submanifolds. According to the strong maximum principle, two 
different conformal solitons
cannot ``touch'' each other at an interior or boundary point; see
\cite{eschenburg}.

\begin{lemma}\label{sc-conv}
Let $\mathcal{S}\subset (\h^{n+1},\g_{\h})$ be a $($Euclidean$)$ spherical cap centered
at a point of $\bi \h^{n+1}$ and $2\le k\le n$ a natural number.
Then $\mathcal{S}\subset(\h^{n+1},\gI(k))$ is strictly convex with respect to the upward pointing normal direction.
Moreover, there is no $k$-dimensional soliton with respect to $-\partial_0$ touching $\mathcal{S}$ from above.
\end{lemma}

\begin{proof}
If $\nu$ is a unit
normal vector field along $\mathcal{S}\subset(\h^{n+1},\g_{\h})$ then $\tilde{\nu}=\lambda^{-1}\nu$ is a unit normal along $\mathcal{S}\subset(\h^{n+1},\gI(k))$. Denoting with $\II_{\g_\h}$ the scalar second fundamental
form of $\SM \subset (\h^{n+1},\g_{\h})$ in the direction of $\nu$, and with $\II_{\gI}$ that of $\SM \subset (\h^{n+1},\gI(k))$ in the direction of $\tilde{\nu}$, from \eqref{secformconf} it follows that
\begin{equation}\label{secconf}
		\II_\gI =  e^\frac{1}{kx_0}  \left\{ \II_{\g_{\h}}
		-\nu\left(\frac{1}{kx_0}\right) \g_{\h} \right\}.
\end{equation}
Let us choose as $\nu$ the upward pointing unit normal along $\SM$.
Since $\SM$ is a totally geodesic hypersurface of $(\h^{n+1},\g_{\h})$, from the last identity we obtain that
$$
\II_\gI =\frac{e^\frac{1}{kx_0}}{k}\g_{\h}(\nu,\partial_0)\g_{\h}>0.
$$
Hence $\SM$ is convex when the ambient space is equipped with the Ilmanen metric.
The last claim of the lemma follows by the strong maximum principle of Jorge \& Tomi \cite{jorge}.
\end{proof}

Now we show that, similarly to minimal submanifolds, solitons satisfy the convex
hull property.

\begin{lemma}\label{prop_convexhull}
Let $M \subset \h^{n+1}$ be a $k$-dimensional, connected and properly immersed soliton with respect to
$-\partial_0$. Then, $\bi M\neq\emptyset$ and $M$ is contained in the cylinder
$\R^+ \times \operatorname{conv}(\bi  M)$, where $\operatorname{conv}$ is the $($Euclidean$)$ convex hull 
in $\bi \h^{n+1}$.  
\end{lemma}

\begin{proof}
Observe at first that $\partial'_{\infty}M\neq\emptyset$ since otherwise we can touch $M$ from below
by a spherical cap, something which contradicts Lemma \ref{sc-conv}. Assume now that
$\operatorname{conv}(\bi M)\not\equiv\bi\h^{n+1}$, because otherwise we have nothing to prove.
Fix a hyperplane $\pi \subset \bi \h^{n+1}$ not intersecting $\bi  M$, and let $\Pi\subset\h^{n+1}$
be the totally geodesic hyperplane with $\partial'_\infty \Pi = \pi$. Denote with $U_\pi$ the open half of
$\h^{n+1}$ such that $\partial U_\pi = \Pi$ and $\bi  U_\pi \cap \bi M = \emptyset$. By the properness of 
$M$, we can pick a sufficiently small spherical barrier $\SM \subset U_\pi$ centered at a point of
$\bi U_\pi$ and lying outside of $M$. Due to Lemma \ref{sc-conv}, we can slide $\SM$ towards $\Pi$ without intersecting $M$, until its boundary at infinity  touches $\pi$ at a point $q$. Denote by $p$ 
the center of $\SM$ and consider
the family of spherical barriers $\{\SM_{\lambda}\}$ of radius $\lambda>0$, passing through $q$ and 
centered in the half-line emanating from $q$ in the direction of $p$. The family foliates $U$ and, again by Lemma \ref{sc-conv}, $M$ does not intersect any spherical cap $\SM_\lambda$. Hence, 
$M\subset \h^{n+1} \backslash U_\pi$. The conclusion follows by the arbitrariness of $\pi$. 
\end{proof}

\subsection{Graphical solitons}
Let us consider solitons of the form
$$
\Gamma(u)=\big\{(u(x);x)\in\h^{n+1}=\R^{+}\times\R^n:x\in \Omega\subset\R^{n}\big\},
$$
where $u\in C^{\infty}(\Omega)$.
\begin{lemma}\label{graph1}
The graph $\Gamma(u)\subset\h^{n+1}$ is a soliton with respect to $-\partial_0$,
if it satisfies the equation
\begin{equation}\label{SE}\tag{$\mathrm{SE}$}
\diver\left(\frac{D u}{\sqrt{1+|D u|^2}}\right)=\frac{-nu-1}{u^2\sqrt{1+|D u|^2}},
\end{equation}
where $\diver$ is the Euclidean divergence, $D$ denotes the Euclidean gradient and $|\cdot|$ the Euclidean norm. In particular, $\Gamma(u)$ has nowhere zero mean
curvature.
\end{lemma}
\begin{proof} The graph is the image of
$\psi:\Omega\to\h^{n+1}$ given by
$\psi(x)=(u(x);x),$
for each point $x\in\Omega$.
As usual, denote by $\g_{\h}$ the metric of $\h^{n+1}$ and by $\nabla$ its Levi-Civita connection.
The components of the induced metric $\g$ on the graph in the basis
$\{\partial_j\}$ are
\begin{equation}\label{indmetric}
g_{ij}=\frac{u_iu_j+\delta_{ij}}{u^2},
\end{equation}
where $i,j\in\{1,\dots,n\}$. Moreover, the components
$g^{ij}$ of the inverse of $\g$ are given by
\begin{equation}\label{g-1}
g^{ij}=u^{2}\left(\delta^{ij}-\frac{u^iu^j}{1+|D u|^2}\right),
\end{equation}
where $u^i = \delta^{ij}u_j$ and $Du = u^j \partial_j$. The unit normal $\nu$ along the graph
is
\begin{equation}\label{n1}
\nu=\frac{u\,\partial_0-uD u}{\sqrt{1+|D u|^2}}=\frac{u\,\partial_0-u u^j \partial_j}{\sqrt{1+|Du|^2}}.
\end{equation}
Making use of the Koszul formula and \eqref{n1}, the components of the second fundamental form are
\begin{eqnarray*}
b_{ij}=\g_{\h}\big(\nabla_{\psi_i}\psi_j,\nu\big)=\g_{\h}(\psi_{ij},\nu)
+u^{-1}\langle\psi_i,\psi_j\rangle\g_{\h}(\partial_0,\nu)
=\frac{uu_{ij}+\delta_{ij}+u_iu_j}{u^2\sqrt{1+|D u|^2}}
\end{eqnarray*}
for each $i,j\in\{1,\dots,n\},$ where $\langle\cdot\,,\cdot\rangle$ stands for the Euclidean standard inner 
product. Using \eqref{indmetric} and raising one index by means of the graph metric, the shape operator satisfies
\begin{equation}\label{eq_shape}
b^k_j = g^{ki}b_{ij} = \frac{1}{\sqrt{1+|D u|^2}} 
\left[ u \left( u^k_j - \frac{u^ku^i u_{ij}}{1+|Du|^2}\right) + \delta^k_j \right].
\end{equation} 
From \eqref{eq_shape}, the unnormalized scalar mean curvature is
\begin{equation}\label{mean1}
H= g^{ij}b_{ij} = u\,\diver\left(\frac{D u}{\sqrt{1+|D u|^2}}\right)+\frac{n}{\sqrt{1+|D u|^2}}.
\end{equation}
One the other hand, $\Gamma(u)$ is a soliton with respect to $X=-\partial_0$ if and only if
\begin{equation}\label{mean2}
H=-\g_{\h}(\partial_0,\nu)=\frac{- 1}{u\sqrt{1+|D u|^2}}.
\end{equation}
Combining \eqref{mean1} with \eqref{mean2} we obtain the desired result.
\end{proof}

\subsection{Sub and supersolutions} Let us describe here special sub and supersolutions
to the quasilinear differential equation \eqref{SE} that we will use often in the rest of the paper.

\begin{definition}
Let $\Omega$ be a domain of the Euclidean space $\R^n$. A $C^2$-smooth function
$u:\Omega\to(0,\infty)$ is called {\em subsolution} $($resp. {\em supersolution}$)$
to \eqref{SE} if it satisfies
$$
\diver \left( \frac{Du}{\sqrt{1+|D u|^2}}\right)
\ge \frac{-nu - 1}{u^2\sqrt{1+|D u|^2}} \qquad\big(\text{\rm resp.}\le\big).
$$
In this case we say that the graph of $f$ is a subsolution  $($resp. {\em supersolution}$)$ to \eqref{SE}.
\end{definition}

\begin{remark}
Euclidean half-spheres in $\h^{n+1}$ whose centers are at $\bi\h^{n+1}$ and Euclidean
half-hyperplanes whose boundaries are at  $\bi\h^{n+1}$  are subsolutions to the equation 
\eqref{SE}. If $u$ is a solution to \eqref{SE} and $\eps>0$,
then $u+\eps$ is a supersolution and $u-\eps$ is a subsolution to \eqref{SE}.
\end{remark}

\subsection{Proof of Theorem \ref{teoexiGR}}
We examine here complete solitons of the form
$\varGamma\times\R^{n-1}\subset\h^{n+1}$,
where $\varGamma$ is a curve in the $x_0x_1$-plane. In regions
where $\varGamma$ can be represented as the image of $\gamma(t)=(u(t),t)$ for some function $u$, equation \eqref{SE} becomes:
\begin{equation}\label{ode_semeno}
	\frac{u''}{1+(u')^2} = \frac{- nu - 1}{u^2}.
\end{equation}
Notice that $u$ is strictly concave. Hence, up to possibly one point (the maximum of $u$, if any), $\varGamma$ can also be rewritten as the union of graphs of the type $\gamma(z) = (z, \phi(z))$ for some functions $\phi$. In this case, \eqref{ode_semeno} rewrites as
\begin{equation}\label{ode_semeno_phi}
\frac{\phi_{zz}}{1+\phi_z^2} = \frac{nz+1}{z^2}\phi_z.
\end{equation}
By \eqref{ode_semeno_phi}, if $\gamma'$ is parallel to $\partial_0$ at some point $(z_0,\phi_0)$ then the 
unique solution is $\gamma(z) = (z, \phi_0)$ and the corresponding soliton is a vertical hyperplane.
Let us treat now the case where $\gamma'$ is nowhere parallel to $\partial_0$. In this case,
$\gamma$ can be globally written as a graph of the type $(u(t),t)$. In particular, the derivative $\phi_z$ 
does not change sign on each maximal interval where $\gamma$ can be written in the form $(z, \phi(z))$. 
Since $\phi_c \doteq \phi+c$ and $\phi^* = -\phi$ still solve \eqref{ode_semeno_phi}, without loss of generality we 
can consider a maximal interval where $\gamma(z) = (z, \phi(z))$ and $\phi_z<0$. Equation 
\eqref{ode_semeno_phi} guarantees that $\phi$ is concave therein, so the interval of
definition of $z$ is of the form $(0,h)$. 
Rewrite \eqref{ode_semeno_phi} as
\[
\frac{\di}{\di z} \int_{-\phi_z}^\infty \frac{ \di t}{t(1+t^2)} = - \frac{n}{z} - \frac{1}{z^2}, 
\] 
and note that the integral can be explicitly computed. Integrating on $[z,b]$ leads to 
\[
\log \sqrt{ 1 + \phi_z(b)^{-2}} - \log \sqrt{1+\phi_z(z)^{-2}} = - n \log \left( \frac{b}{z}\right) + \frac{1}{b} - \frac{1}{z}.
\]
If $h = \infty$, letting $b \to \infty$ and recalling that $\phi_z$ is negative and decreasing we easily get a contradiction. Hence, $h< \infty$, and by concavity and maximality of $h$ we shall have $\phi_z(h^-) = -\infty$. Letting $b \to h$ we get
\[
\log \sqrt{1 + \phi_z(z)^{-2}} = n \log \left( \frac{h}{z}\right) + \frac{1}{z} - \frac{1}{h}.
\]
In particular, $\phi_z(0^+)=0$, i.e., $\gamma$ meets the $x_1$-axis orthogonally. Extracting $\phi_z$ and integrating once more on $[z,b]$, we get
\[
\phi(z)-\phi(b) = \int_z^b \left\{ \left( \frac{h}{t}\right)^{2n} e^{\frac{2}{t} - \frac{2}{h}} - 1 \right\}^{-\frac{1}{2}} \di t.
\]
We deduce that $\phi(h)$ is finite, and up to translation we can assume that $\phi(h) = 0$. Writing $\gamma$ as a graph of type $(u(t),t)$, noting that $u^*(t) \doteq u(-t)$ still solves \eqref{ode_semeno} and $u'(0)=0$, we conclude that $u$ is even and $u(t) = \phi^{-1}(t)$ for $t>0$ has the behaviour described in Theorem \ref{teoexiGR}. Also, letting $b \to h$ we obtain
\begin{equation}\label{eq_phi_h}
\phi(z) = \int_z^h \left\{ \left( \frac{h}{t}\right)^{2n} e^{\frac{2}{t} - \frac{2}{h}} - 1 \right\}^{-\frac{1}{2}} \di t = h \int_{z/h}^1 \left\{ s^{-2n} e^{\frac{2-2s}{hs}} - 1 \right\}^{-\frac{1}{2}} \di s.
\end{equation}
Hence, for $h$ varying in $\R^+$, the functions $\phi$ in \eqref{eq_phi_h} provide a foliation of $\{x_1 > 0\}$. This concludes the proof of Theorem \ref{teoexiGR}.

\subsection{Rotationally symmetric solitons} We deal now with solitons for
$-\partial_0$ which are rotationally  symmetric with respect to the $x_0$-axis. Consider a curve
$\gamma(t)=(x_0(t),x_1(t))$ in the $x_0x_1$-plane, defined for $t\in I\subset\R$ and contained in
the set $\{(x_0,x_1)\in\R^2:x_0>0,x_1> 0\},$
and let us rotate it around the
$x_0$ axis.
Locally, $\gamma$ can be written either as a graph of a function $u$ over the
$x_1$-axis or as a graph of a function $\phi$ over the $x_0$-axis. Let us describe the corresponding soliton
equations for $u$ and $\phi$.

{\bf Case 1:} The curve $\gamma$ is written as a graph of the form $\rho\mapsto(u(\rho),\rho)$ defined on an
interval $(r_1,r_2)$. In this case, the obtained hypersurface is a graph over the annulus
$$
A =\big\{(x_1, \dots, x_n) \in \partial_\infty' \h^{n+1}: r_1 < |x|=(x_1^2+\cdots+x_n^2)^{1/2}< r_2 \big\},
$$
and it can be parametrized by the embedding $R_u:A\to\R^+\times\R^n$ given by
$$
R_u(x_1,\dots,x_n)=(u(|x|);x_1,\dots,x_n).
$$
Then $R_u$ is a conformal soliton with respect to $-\partial_0$ if and only if $u$ solves
\begin{equation} \label{ODEAnnulus}
\frac{u''}{1+(u')^2}+\frac{n-1}{\rho} u'=- \frac{1+nu}{u^2}, \qquad \rho=|x|\in(r_1,r_2).
\end{equation}
{\bf Case 2:} The curve $\gamma$ is written as a graph of the form $z\to(z,\phi(z))$ defined on an
interval $(z_1,z_2)$. In this case
the hypersurface can be parametrized as a graph over the cylinder $\R^+\times\Smmenosum\subset\R^+\times\R^{m}$. To state the equation for $\phi$, we first locally identify $\h^{n+1}=\R^+\times\R^n$ with $\R^+ \times \R^+ \times \mathbb{S}^{n-1}$
via the local diffeomorphism
$$
(x_0;x_1,\dots,x_n)\to(z,\rho,\omega)=\left(x_0;\sqrt{x_1^2+\dots+x^2_n};\frac{(x_1, \dots, x_n)}{\sqrt{x_1^2+\dots+x^2_n}}\right).
$$
The hyperbolic metric in the new coordinate system has the form
$$
\g_{\h}=z^{-2}\big(\di z^2 + \di \rho^2 +\rho^2 \g_{\mathbb{S}}\big),
$$
where $\g_{\mathbb{S}}$ is the standard round metric of $\mathbb{S}^{n-1}$. We can parametrize the hypersurface via the map
$
C_{\phi}:(z_1,z_2) \times \mathbb{S}^{n-1} \to \R^+ \times \R^+ \times \mathbb{S}^{n-1}
$
given by
$$C_{\phi}(z,\omega)=(z,\phi(z),\omega).$$
Then $C_{\phi}$ produces a soliton of the MCF if and only if $\phi$ satisfies the ODE:
\begin{equation} \label{ODECylinder}
\frac{z\phi''}{{1+(\phi')^2}} - \frac{1+nz}{z}\phi' -\frac{(n-1)z}{\phi} = 0,\qquad z\in(z_1,z_2),
\end{equation}

(we use $\phi'$ instead of $\phi_z$ for notational convenience). In the next lemma we examine the qualitative behaviour of solutions to \eqref{ODECylinder}. 

\begin{lemma}\label{lem_1_phi}
Let $z_0\in(z_1,z_2)$ with $\phi(z_0)>0$. The following facts hold:
\begin{enumerate}[\rm (1)]
\item{\bf (Concave branch)} If $\phi'(z_0) < 0$ and $\phi''(z_0) \le 0$, then $\phi$ can be extended to the interval $[0,z_0]$. Moreover, $\phi'<0$, $\phi'' < 0$ on $(0,z_0)$ and 
\[
\phi(z) = \phi(0) - \frac{n-1}{3\,\phi(0)}z^3 + o(z^3) \quad \text{as } \, z \to 0. 
\] 
\item
{\bf (Convex branch)} If either $\phi'(z_0) \ge 0$ or $\phi''(z_0) > 0$, then there exists an interval
$[\lambda_0,h)\subset(0,\infty)$ containing $z_0$ where $\phi$ can be defined, is positive and satisfies
\[
\begin{array}{ll}
(i) : & \phi''>0 \ \ \ \text{on } \, (\lambda_0,h), \quad \phi'(\lambda_0)< 0 \quad
\text{and}\quad \phi''(\lambda_0) = 0. \\[0.2cm]
(ii) : & \disp \lim_{z \to h} \phi(z) \in\R^+ \quad\text{and}\quad \lim_{z \to h} \phi'(z) = \infty.
\end{array} 
\]
In particular, $\phi$ has a unique minimum on $(\lambda_0,h)$. 
\end{enumerate}
\end{lemma}

\begin{proof}
$(1)$ Let $(z_1,z_0]$ be the maximal interval on the left of $z_0$ where $\phi$ is defined. By \eqref{ODECylinder}, if $\phi'(z) = 0$ for some $z \in (z_1,z_0)$ then $\phi''(z) > 0$. Hence, from $\phi'(z_0)<0$ we deduce that $\phi'<0$ on $(z_1,z_0]$. We claim that $\phi'' \le 0$ on $(z_1,z_0]$. This fact is true since otherwise, from $\phi''(z_0) \le 0$
we deduce that there exist an interval $(a,b] \subset (z_1,z_0]$ such that $\phi''>0$ on $(a,b)$ and $\phi''(b)=0$. Then, from \eqref{ODECylinder} we have
\[
-\phi'(a) \left(n + \frac{1}{a}\right) \le (n-1) \frac{a}{\phi(a)}\quad\text{and}\quad
-\phi'(b) \left(n + \frac{1}{b}\right) = (n-1) \frac{b}{\phi(b)}.
\]
Now from $0 < \phi(b) < \phi(a)$,  we obtain that $0 > \phi'(a) > \phi'(b)$, contradicting that $\phi''>0$ on $(a,b)$. To show 
the strict inequality $\phi''<0$ on $(z_1,z_0]$, observe that any point $z \in (z_1,z_0]$ for which $\phi''(z) = 0$ must
be a local maximum for $\varphi''$, hence a point for which $\varphi'''(z) = 0$. Differentiating 
\eqref{ODECylinder}, and evaluating at $z$, we have
\begin{eqnarray*}
0 & = & \disp \frac{z\phi'''(z) + \phi''(z)}{{1+(\phi'(z))^2}} - \frac{2 z \phi'(z)(\phi''(z))^2}{(1+(\phi'(z))^2)^2}
 -\disp \frac{1+mz}{z}\phi''(z) + \frac{\phi'(z)}{z^2} -\frac{(n-1)}{\phi(z)} + \frac{(n-1)z\phi'(z)}{\phi^2(z)} \\
& = & \disp \frac{\phi'(z)}{z^2} -\frac{(n-1)}{\phi(z)} + \frac{(n-1)z\phi'(z)}{\phi^2(z)} < 0,
\end{eqnarray*}
which gives a contradiction. From $\phi'<0$ and $\phi''<0$ we deduce that the limit $\lim_{z \to z_1} \phi'(z)$ exists 
and is finite. If $z_1>0$, then the function $\phi$ can be extended beyond $z_1$, which contradicts the maximality
of the interval. Hence, $z_1 = 0$ and  so $\phi$ has a $C^1$-extension up to $z=0$. By using spherical barriers and 
the strong maximum principle, we deduce that $\lim_{z\to 0}\phi'=0$. To find the asymptotic
behaviour of
$\phi$, consider the function
\begin{equation}\label{formulai}
F(z)= \frac{z^{-n}\phi'(z)}{\sqrt{1+(\phi'(z))^2}},
\end{equation}
defined in a neighbourhood $(0,a]$.
Differentiating and using \eqref{ODECylinder}, we deduce that
$$
F' - \frac{F}{z^2} = \frac{n-1}{z^n \phi \sqrt{1+ (\phi')^2}},
$$
that is,
$$ \left(F \eint\right)' =\frac{(n-1)\eint}{z^n \phi \sqrt{1+(\phi')^2}}.
$$
Integrating on $[z,a]$, we get
\begin{equation}\label{orioF}
F(a) - F(z) e^{-\frac{1}{a}+ \frac{1}{z}} =\int^{a}_z  \frac{(n-1)e^{-\frac{1}{a}+\frac{1}{r}}}{r^n \phi(r)\sqrt{1+ \phi'(r)^2}} \di r,
\end{equation}
or, equivalently,
\[
F(z) = F(a)e^{\frac{1}{a}- \frac{1}{z}} - e^{-\frac{1}{z}} \int^a_z  \frac{(n-1)e^{\frac{1}{r}}}{r^n \phi(r)\sqrt{1+ \phi'(r)^2}}\di r.
\]
Computing the asymptotic behaviour of the right hand side as $z \to 0$ and using the facts
$$\lim_{z\to 0}\phi(z) = \phi(0) > 0,\quad \lim_{z\to 0}\phi'(z)=0\quad\text{and}\quad F(z) \sim z^{-n} \phi'(z)$$
we infer
\[
\phi'(z) = - \frac{n-1}{\phi(0)}z^2 + o(z^2), \quad \text{as } \, z \to 0.
\]
By integration we obtain the desired asymptotic behaviour of $\phi$ close to $z=0$.

(2) Notice that if $\phi'(z_0) \ge 0$ then \eqref{ODECylinder} implies $\phi''(z_0) >0$. Let $(\lambda_0,h)$ be the maximal interval containing $z_0$ where $\phi''>0$. By convexity,
the limits of $\phi$ and $\phi'$ as $z$ tends to $h$ exist
and
\begin{equation}\label{phih}
\lim_{z\to h}\phi(z)=\phi_h\in[0,\infty] \quad\text{and}\quad
\lim_{z \to h} \phi'(z) =\phi_h'\in( -\infty,\infty].
\end{equation}
{\bf Claim 1:} {\em It is not possible that $h,\phi_h'$ are finite and $\phi_h>0$.} 

{\em Proof of the claim.} Indeed, if this is possible then $\phi$ can be extended beyond $h$ and therefore, by the maximality of $h$, $\phi''(h) = 0$. Equation \eqref{ODECylinder} then gives $\phi'(h)<0$ and by part $(1)$, the function $\phi$ would be a concave branch before $h$, which is a contradiction.  \hfill{$\circledast$}

Observe that $\phi''>0$ implies that $\phi'$ does not change sign 
in some interval $(\lambda,h)$.

{\bf Claim 2:} {\em $\phi$ attains a minimum on $(\lambda_0,h)$ and the minimum is unique.}

{\em Proof of the claim.} 
Uniqueness is immediate from $\phi''>0$. Because of the convexity, for the existence we only have
to exclude the possibility that $|\phi'|>0$ on $(\lambda_0,h)$.   
\begin{enumerate}[\rm(a)]
\item First, we rule out the possibility that $\phi'<0$ on $(\lambda_0,h)$. Suppose to the contrary that
this case occurs. If $h=\infty$,
then the convexity and the positivity of $\phi$ imply that
$\lim_{z\to\infty}\phi'(z)= 0.$
From \eqref{ODECylinder} we deduce that the existence of a positive constant $C$
such that
$\phi''(z)>C$
for large values of $z$. Integrating, $\phi'(z)\to\infty$, a contradiction. Assume 
now that $h<\infty$ and the limit
$\phi_h$ in  \eqref{phih} vanishes. By convexity, fixing $\lambda \in (\lambda_0,h)$ there exists a 
constant $C>0$ such that 
$\phi(z)\le C(h-z)$, on $[\lambda,h).$
From \eqref{ODECylinder} and the fact that $\phi'$ is bounded on $[\lambda,h)$, it follows that there exist positive constants $C_1$ and $C_2$ such that
$$
\phi''\ge \frac{C_1}{h-z}-C_2 \quad \text{on } [\lambda,h). 
$$
Integrating on $[\lambda,z]$ we get
$$
\phi'(z)-\phi'(\lambda)\ge -C_1\log\left(\frac{h-z}{h-\lambda}\right)-C_2(z-\lambda)\to \infty
\quad\text{as}\quad z\to h,
$$
contradicting $\phi'(z)<0$. Consequently, $h$ is finite and $\phi_h\in\R^+$. From \eqref{phih} and $\phi'<0$ we deduce that $\phi'_h$ is finite as well. By Claim 1, this leads to a contradiction. 
\smallskip
\item We also rule out the possibility that $\phi'>0$ on $(\lambda_0,h)$. Arguing again by contradiction,
let us suppose that this happens. We may represent the graph
of $\phi$ on $(\lambda_0,h)$ as the graph of a function $u$ over the $x_1$-axis defined on an interval
$(r_1,r_2) \subset (0,\infty)$, where
$r_1=\phi(\lambda_0^+)$ and $ r_2=\phi(h^-).$
Consider $\Phi:(r_1,r_2)\to\R$ given by
\begin{equation}\label{Phifunc}
\Phi(\rho)=\frac{u'(\rho)\rho^{n-1}}{\sqrt{1+(u'(\rho))^2}}.
\end{equation}
From \eqref{ODEAnnulus} we deduce that
$$
\Phi'(\rho)= -\frac{(1+nu(\rho))\rho^{n-1}}{u^2(\rho)\sqrt{1+(u'(\rho))^2}} <0\qquad
\text{for}\quad r\in(r_1,r_2).
$$
Hence $\Phi$ is strictly  decreasing. Since $u$ is increasing, $\Phi$ is positive on $(r_1,r_2)$.
If $r_1=0$, then letting $\rho\to 0$, we get that $\Phi(0^+)=0$ which contradicts
the aforementioned properties of $\Phi$. Consequently $r_1>0$. Moreover, $\lambda_0$ must be positive.
Indeed, otherwise $\phi$ is defined, convex and increasing on $(0,h)$.
Consider the  functional $F$ given by \eqref{formulai} on the interval $(0,a] \subset (0,h)$. From the
properties of $\phi$, it follows that $\phi'$ is bounded on $(0,a]$. Hence,
$$
\lim_{z\to 0} \int^{a}_z  \frac{(n-1)e^{-\frac{1}{a}+\frac{1}{r}}}{r^n \phi(r)\sqrt{1+ (\phi'(r))^2}} \di r =\infty.
$$
From \eqref{orioF} we now deduce that
$$
\lim_{z\to 0}F(z) e^{-\frac{1}{a}+ \frac{1}{z}}=-\infty.
$$
Consequently $F<0$ in a neighbourhood of $z=0$. Since $F$ and $\phi'$ have the same sign, this leads to a contradiction. Thus $\lambda_0$ must be positive. Because $\phi(\lambda_0)>0$ and
$\phi'(\lambda_0)\ge 0$ we can extend the solution below $\lambda_0$.  The minimality of $\lambda_0$
implies that $\phi''(\lambda_0)=0$ and from \eqref{ODECylinder} we obtain that $\phi'(\lambda_0)<0$,
contradiction.  \hfill{$\circledast$}
\end{enumerate}

{\bf Claim 3:} {\em $\lambda_0>0$ and $(i)$ holds.}

{\em Proof of the claim.} 
Suppose that this is not true. Recalling that $\phi''>0$ on $(\lambda_0,h)$, one of the following holds:
\begin{enumerate}
\item[\rm(a)] $\lambda_0=0$;
\smallskip
\item[\rm(b)] $\lambda_0>0$ and $\lim_{z\to\lambda_0}\phi(z)=\infty$;
\smallskip
\item[\rm(c)] $\lambda_0>0$ and $\lim_{z\to\lambda_0}\phi(z)<\infty$ and $\lim_{z\to\lambda_0}\phi'(z)=-\infty$.
\end{enumerate}
Cases (a) and (b) are excluded by considering the soliton $M$ obtained by rotating the graph of $\phi$
and sliding-enlarging a small spherical barrier below $M$ up to a touching point. As for (c),
writing the graph in terms of the parametrization $\rho\to(u(\rho),\rho)$, we get that $u$ can be extended
near $\phi(\lambda_0)$ with zero derivative and nonnegative second derivative. This contradicts
\eqref{ODEAnnulus} and concludes the proof of the claim. \hfill{$\circledast$}

{\bf Claim 4:} {\em $h$ is finite and  $(ii)$ holds.}

{\em Proof of the claim.} 
By Claim 2, $\phi'$ and $\phi''$ are positive in an interval $(\lambda,h)$. Then we represent the graph of
$\phi$
therein as a graph of the form $\rho\mapsto(u(\rho),\rho)$
satisfying $u'>0$ and $u''<0$ for $\rho\in(r_1,r_2)$, where $r_2=\phi(h^-)$. First we show that $r_2<\infty$. Suppose that
this is not the case and denote by $u_{\infty}$ the limit of $u$ as $\rho$ tends to $\infty$. If $u_{\infty}$
is finite, then $u'\to 0$ as $\rho\to\infty$ and
there exists a sequence $\{\rho_j\}_{j\in\n}$ tending to infinity
such that $u''(\rho_j) \to 0$. Evaluating \eqref{ODEAnnulus} at $\rho_j$ and passing to the limit we easily get a contradiction. If $u_{\infty}=\infty$, 
consider $\Phi$ given in \eqref{Phifunc} for $\rho\in(r_1,\infty)$.
Fixing $\rho_0>1$, from the monotonicity of
$\Phi$ there exists a constant $C>0$ such that
\[
\frac{u'(\rho)}{\sqrt{1+(u'(\rho))^2}} \le C \rho^{1-n} \qquad \text{for all}\quad \rho \ge \rho_0,
\]
from where we deduce that
\[
u'(\rho) \le \frac{C \rho^{1-n}}{\sqrt{1-C^2 \rho^{2-2n}}}.
\]
If $n\ge 3$, then by integration we get that $u_{\infty}<\infty$, contradiction. If $n = 2$, integration gives
$$
u'(\rho) \le {C_1}\rho^{-1}\,\,\, \text{and}\,\,\, u(\rho) \le C_1 \log \rho
\qquad\text{for some constant}\,\, C_1>0.
$$
Inserting into \eqref{ODEAnnulus}, we conclude that there exists a constant $C_2>0$ such that
$$
u'' \le -\frac{C_2}{\log \rho} \qquad \text{for large enough}\,\, \rho.
$$
By integrating we see that $u'(\rho) \to -\infty$ as $\rho \to \infty$, which gives a
contradiction. Hence $r_2<\infty$, which implies that $h$ is finite by the concavity of $u$.
Moreover, $u'\to\ell\ge 0$ as $\rho\to r_2$, whence $u$ can be extended beyond $r_2$. From
\eqref{ODEAnnulus} we get $u''(r_2)<0$. If $\ell>0$, then rephrasing the curve in terms of $\phi$,
it follows that $\phi$ can be extended beyond $h$ to a convex function, contradicting the maximality
of $h$. Thus, $\ell=0$ and $(ii)$ follows.  \hfill{$\circledast$}

This completes the proof of the lemma.
\end{proof}

\begin{lemma}\label{lemmaxgraph}
Let $u_1$ and $u_2$ be solutions to \eqref{ODEAnnulus} on an interval $(r_1,r_2)$. Then, either
$u_1\equiv u_2$ or $u_1-u_2$ does not have a non-negative local maximum on $(r_1,r_2)$.
\end{lemma}
\begin{proof}
Assume to the contrary that $u_1-u_2$ attains a local maximum $c_0\ge 0$ at a point $r_0$.
Then, $u_1\le u_2+c_0$
near $r_0$ with equality attained at the point $\rho_0$. From $c_0\ge 0$ it follows that
$u_2+c_0$ is a
supersolution to \eqref{ODEAnnulus} and this contradicts the strong maximum principle.
\end{proof}

Now we are ready to state and prove the main theorem characterizing all rotationally symmetric
conformal solitons.

\begin{theorem}\label{preD}
There are exactly two families of complete solitons with respect to $-\partial_0$ which are rotationally
symmetric around the $x_0$-axis. They are properly embedded and, denoting with 
$$\gamma\subset \{(x_0,x_1)\in\R^2:x_0>0,x_1\ge 0\}$$
the rotated curve, one of the following cases occurs:
\begin{enumerate}[\rm(1)]
\item{\bf (Winglike catenoids)} Suppose that $(x_0\circ\gamma)'(t_0)=0$ at some interior point $t_0$
and let $\gamma(t_0)=(h,R)$, $R>0$. Then the curve $\gamma$ can be written as the bi-graph
over the $x_0$ axis of $\phi_1,\phi_2 : (0,h] \to (0,\infty)$
satisfying the following properties:
\smallskip
\begin{itemize}
\item[(a)] It holds $\phi_1 < \phi_2$ on $(0,h)$ and $\phi_1(h)=\phi_2(h) = R$. Furthermore, $\phi_1(0^+) < \phi_2(0^+)$, namely, $\gamma$ cannot have the same end-points; 
\smallskip
\item[(b)] the graph of $\phi_2$ is a concave branch on $(0,h)$;
\medskip
\item[(c)] there exists $\lambda_0\in(0,h)$ such that $\phi_1$ is the union of a concave branch on $(0,\lambda_0)$ and a
convex branch on $(\lambda_0,h)$;
\smallskip
\end{itemize}
\smallskip
\item{\bf(Bowl solitons)} If $x_0\circ\gamma$ does not have interior stationary points,
then $\gamma$ is the graph of $\phi : (0,h] \to [0,\infty)$ satisfying the following properties:
\smallskip
\begin{itemize}
\item[(a)] The graph of $\phi$ is a concave branch on $(0,h)$;
\smallskip
\item[(b)] it holds $\phi(h) = 0$, $\phi'(h^-) = \infty$.  
\end{itemize} 
\end{enumerate}
\end{theorem}
\begin{proof}
(1) Writing $\gamma$ near $(h,R)$ as a graph of the form $\rho\to(u(\rho),\rho)$, the 
equation \eqref{ODEAnnulus} and $u'(R)=0$ gives $u'' < 0$ near $R$. In particular, $u$ is 
decreasing after $R$ and increasing before. Hence, by Lemma \ref{lem_1_phi}, for $\rho > R$ the 
graph of $u$ extends to a concave branch of the form
$z\to(z, \phi_2(z))$ for $\phi_2: (0,h) \to \R^+$, while for $\rho<R$ it extends to a convex branch of 
a function $\phi_1: (\lambda_0,h) \to \R^+$. At the point $(\lambda_0,\phi_1(\lambda_0))$ we can 
apply Lemma \ref{lem_1_phi} again to deduce that $\phi_1$ extends to a concave branch on 
$(0,\lambda_0)$.
Set $c = \min \phi_1$. It remains to prove that $\phi_1< \phi_2$ on $(0,h)$ and
$\phi_1(0^+) < \phi_2(0^+).$
Arguing by contradiction, if any of the properties fails then, according 
to what we already proved, there exists $R_1 \in (c, \phi_2(0^+)]$ such that
the curve $\gamma$ can be 
written as a bigraph of functions $u_i : (c,R_1)\to \R^+$ over the $x_1$-axis satisfying
$u_1 > u_2$ on $(c,R_1)$, $u_1-u_2 \to 0$ as $\rho \to c$ and as $\rho \to R_1$. This contradicts 
Lemma \ref{lemmaxgraph}. 

(2)
If $x_0\circ\gamma$ has no stationary points, then $\gamma$ can be globally written as a graph of the form
$z\to(z,\phi(z))$. By Lemma \ref{lem_1_phi}, the graph of $\phi$ is the concave branch
defined in a maximal domain $(0,h)$. Representing $\gamma$ as a graph of the form
$\rho\to(u(\rho),\rho)$, concavity implies that $u'(0^+)$ exists and is non-positive. If this value is negative, 
then by inspecting \eqref{ODEAnnulus} for small enough $\rho$ we arrive at a contradiction. Hence, 
$u'(0^+)=0$. It remains to prove that a curve of type (2) actually exists. This is addressed in the next lemma.
\end{proof}


\begin{lemma}\label{r2est}
Fix $R\ge 0$. For $h>0$, the solution $u_h$ to the following problem exists and is unique:
\begin{equation}\label{eq_problem}
\left\{
\begin{aligned}
&\frac{u''}{1+(u')^2}+\frac{n-1}{\rho} u'=- \frac{1+nu}{u^2} \quad\text{for}\quad \rho> R,\\
&\,\,u(R^+)=h,\\
&\,\, u'(R^+)=0.
\end{aligned} \right.
\end{equation}
Moreover, $u_h$ is concave and strictly decreasing. Let $[R,r_2(h))$ be the maximal interval where $u_h$ is defined. Then, for $h \in \R^+$ the graphs of $\{u_h\}$ foliate the region $\{\rho > R\}$, and $r_2:(0,\infty)\to(R,\infty)$ is a strictly increasing bijection.
\end{lemma}

\begin{proof} 
Uniqueness for $u_h$ is a consequence of Lemma \ref{lemmaxgraph} and the strong maximum principle (or the Hopf Lemma, if $R>0$) for the 
solitons $M_h$ obtained by rotating the graphs $\{x_0 = u_h(|x|)\}$ around the $x_0$-axis. Note 
that, if $R=0$, $M_h$ is $C^1$ near the origin and $\gI(n)$-minimal, hence it is smooth therein. Existence for 
\eqref{eq_problem} is standard if $R>0$. By adapting the proof of Lemma \ref{lem_1_phi}, one easily sees 
that $u_h$ is strictly decreasing for $R>0$ and concave for $\rho>R$. To prove 
existence for $R=0$ one may proceed by adapting the techniques in
\cite[Chapter 5, Theorem 5.10]{bianchini}. However, we give a geometric and simpler proof which exploits the 
results we showed so far. Consider a decreasing sequence $\eps_i \downarrow 0$, and for each $i$ let $u_i$ 
solve
$$ \left\{
\begin{aligned}
&\frac{u_i''}{1+(u_i')^2}+\frac{n-1}{\rho} u_i'=- \frac{1+nu_i}{u_i^2},\\
&\,\,u_i(\eps_i)=h,\\
&\,\, u'_i(\eps_i)=0,
\end{aligned} \right.
$$
on its maximal interval $[\eps_i,R_i)$. Since
$u_i''(\eps_i)=-(1+nh)h^{-2}<0$, from Lemma \ref{lem_1_phi} we get that the graph of $u_i$
is a concave branch in $(\eps_i,R_i)$ and in particular $R_i$ is finite. We claim that $R_i$ is bounded
from below away from zero. Suppose to the contrary that, along a subsequence, $R_i\to 0$. According to
Lemma \ref{lem_1_phi} the graph of $u_i$ is part of a winglike catenoind $M_i$ of height $h$.
Take a grim-reaper cylinder $\mathscr{G}_{s}$ of height $s>h$ and passing through
$(s,0,\dots,0)$. Then, for fixed $R_i$ small enough, $M_i$ lies below $\mathscr{G}_{s}$. Reducing
$s$ up to a first touching point with $M_i$ we reach a contradiction. Let now $0<R^*=\inf R_i$.
By \eqref{ODEAnnulus} and since each $u_i$ is decreasing, the sequence $\{u_i\}$ has uniformly bounded
$C^2$-norm on any fixed compact set of $(0,R^*)$. Therefore, up to a subsequence, $u_i\to u$ in
$C^2_{\rm loc}((0,R^*))$, where $u$ solves $\eqref{ODEAnnulus}$ on $(0,R^*)$. Since each $u_i$ is
decreasing and concave, $u$ is concave, non-increasing and $u(0^+)=h$. As a matter of fact, $u'<0$
by \eqref{ODEAnnulus} and so $u$ is a concave branch. In particular, by the first part of the proof of Theorem \ref{preD}(2), $u'(0^+)=0$.

We next address the properties of $u_h$ for varying $h$. 
\begin{itemize}
\item[(i)]
First we show that $r_2$ is strictly increasing and that $\{u_h\}$ is increasing in $h$. Suppose to the contrary that for
$h_1>h_2$ we have $r_2(h_1)\le r_2(h_2)$. Then, $c \doteq \max (u_1-u_2)$ is positive and attained at some $r \in [R, r_2(h_1))$. Consider the functions $v_j : B_{r_2}(h_1)\backslash B_R \to \R$ obtained by rotating $u_{h_j}$ along the $x_0$-axis, and notice that $u_1$ is a soliton for $-\partial_0$ while $v_2+c$ is a supersolution for \eqref{QQ}, equivalently, it is mean convex for $\gI(n)$ in the downward direction. The interior strong maximum principle (if $r > R$ or $r=R=0$) or the boundary maximum principle (if $r=R>0$) in \cite{eschenburg} imply $v_1 \equiv v_2+c$ wherever both are defined, contradiction.\\
The very same reasoning also proves that $\{u_h\}$ is an increasing family in $h$.  
%
%
%
\medskip
\item[(ii)]
Fix $h_0>0$ and suppose to the contrary that
$$
r_*=\inf\{r_{2}(h):h>h_0\}>r_2(h_0)=r_0.
$$
Then a spherical barrier $\SM$ with radius $(r_*-r_0)/2$ centered at the point
$(0,(r_*+r_0)/2,0,\dots,0)$ fits between the sequence of solitons $\{M_h\}$ that are
generated by the sequence
$\{u_{h}\}$ for $h>h_0$ and $u_{h_0}$. Consider now the sequence
$\{\overline{u}_h=u_h|_{[R,r_0)}\}$ for $h>h_0.$
Observe that $\{\overline{u}_h\}$ uniformly converges to $u_{h_0}$ as $h$
tends to $h_0$. Hence, for $h_1$ sufficiently close to $h_0$ and $\rho_1$ sufficiently
close to $r_0$, we have that
$$u_{h_1}(\rho_1)<\frac{r_*-r_0}{2}.$$
Hence $M_{h_1}$ intersects $\SM$, contradiction. This proves that
$r_2$ is continuous.
\medskip
\item[(iii)] We show that $r_2(h) \to \infty$ as $h \to \infty$. By contradiction, if $r_2(h) \le \ell$ for some $\ell \in \R^+$ and all $h>0$, consider a grim-reaper cylinder $\mathscr{G}$ of width $4\ell$ and symmetric with respect to $\{x_1=0\}$, and fix $h < \sup_{\mathscr{G}}x_0$. Let $M$ be the hypersurface obtained by rotating $u_h$ with respect to the $x_0$-axis. Then, by construction and since $u_h'(R^+)=0$, we can find a grim-reaper cylinder $\mathscr{G}'$ in the foliation determined by $\mathscr{G}$ that touches $M$ from above at some interior point $p$, contradiction. 
\medskip
\item[(iv)] To show that the graphs of $u_h$ foliate $\{\rho> R\}$, let $\rho_0>R$ and by (iii), let $h_0$ satisfy $r_2(h_0) = \rho_0$. It is enough to prove that the map $\eta :  h \in (h_0,\infty) \mapsto u_h(\rho_0) \in (0,\infty)$ is a continuous bijection. First, by (i) $\eta$ is injective and strictly monotone. Next, if $h_j \to h> h_0$ then $\{u_{h_j}\}$ has uniformly bounded $C^0$ (hence, $C^2$) norm on any fixed compact of $(0,\rho_0]$. Up to a subsequence and by the uniqueness of solutions to \eqref{eq_problem}, $u_{h_j} \to u_h$ in $C^2_{\rm loc} ((0, \rho_0])$, thus $\eta$ is continuous. Whence $\eta((h_0,\infty))$ is an interval. By concavity and (iii), $\eta(h) \to \infty$ as $h \to \infty$. On the other hand, if $\delta = \inf \eta((h_0,\infty))>0$, then for $h_j \downarrow h_0$ we would have $u_{h_j}(\rho_0) > \delta$ for each $j$. Let $\rho_1 < \rho_0$ be such that $u_{h_0}(\rho_1) < \delta/2$. From $u_{h_j}(\rho_1) \to u_{h_0}(\rho_1)$, for $j$ large we would have $u_{h_j}(\rho_1)< u_{h_j}(\rho_0)$, contradiction.
\end{itemize}
This completes the proof of the lemma.
\end{proof}

As an immediate consequence of Theorem \ref{preD} and Lemma \ref{r2est}, we are ready for the 

\subsection{Proof of Theorem \ref{TBowl}}
Hereafter, all spheres and balls we use are meant to be centered at the center of $B_R$. By Theorem \ref{preD} there exists a unique graphical rotationally symmetric bowl soliton $\mathscr{B}_R$ having as boundary at infinity $\partial B_R$. Suppose now that $M$ is a properly immersed soliton with respect to $-\partial_0$ such that $\bi M= \partial B_R$. We take two bowl solitons $\mathscr{B}_{r_1}$ and $\mathscr{B}_{r_2}$, with $\bi\mathscr{B}_{r_j}= \partial B_{r_j} \subset\bi\h^{n+1}$ and lying
above and below $M$, respectively. Shrinking $\mathscr{B}_{r_2}$ and enlarging $\mathscr{B}_{r_1}$,
by the maximum principle we deduce that the soliton $M$ must coincide with $\mathscr{B}_R$.

\section{The Plateau problem at infinity}\label{sec4}

In this section we will show the existence of solitons $M \subset \mathbb{H}^{n+1}$ with 
respect to $-\partial_0$ that have prescribed asymptotic boundary values on
$\partial_\infty \mathbb{H}^{n+1}$.

\subsection{Plateau's problem at infinity}
Consider an embedded $(k-1)$-dimensional (topological) submanifold
$\Sigma \subset \partial_\infty \mathbb{H}^{n+1}$, where $2 \le k \le n$. Our aim is to find a 
$k$-dimensional conformal soliton $M$ with respect to the direction $-\partial_0$ whose boundary at infinity satisfies
$\partial_\infty M = \Sigma$. Viewing solitons as minimal submanifolds with respect to the Ilmanen 
metric, our problem can be rephrased as the classical Plateau's problem at infinity for area minimizing submanifolds. Thanks to various works,
the solvability theory for Plateau's problem is 
well understood on Cartan-Hadamard manifolds, that is, on complete, simply connected manifolds  
$N$ with non-positive sectional curvature; see for example 
\cite{anderson1,anderson2,anderson3,bonorino,casteras,casteras2,Teli}.

It is known that every Cartan-Hadamard manifold $N$ can be compactified by adding a sphere at
infinity $\partial_\infty N$; see for details \cite{eberlein}. Given a compact $(k-1)$-dimensional 
submanifold $\Sigma \subset \partial_\infty N$, Plateau's problem at infinity
asks whether exists a $k$-dimensional area minimizing submanifold $M\subset N$ such that
$\partial_\infty M = \Sigma$. It is well-known that Plateau's problem is solvable if $\partial_\infty N$ satisfies certain convexity conditions.  Let us recall here a convenient one proposed by Ripoll \& Telichevesky
in \cite{ripoll}. 

\begin{definition}[SC condition] A Cartan-Hadamard manifold $N$ satisfies the strict convexity condition
{\rm (}{\em SC} condition for short\,{\rm )} at a point  $x \in \partial_{\infty} N$ if, for each relatively open subset
$W \subset \partial_{\infty} N$ containing $x$, there exists a $C^2$-open subset $\Omega \subset N$ such that
$x \in {\rm int}(\partial_{\infty} \Omega) \subset W $ and $N \textbackslash \Omega$ has convex boundary in the inward pointing direction. We say that $N$ satisfies the {\em SC} condition if this holds at every $x \in \partial_{\infty} N$.  
\end{definition}

For a reason to be explained below, Theorem \ref{teoPlateau} is restricted to hypersurfaces. In this case, building on a previous result of Lang \cite{lang}, Cast\'eras, Holopainen \& Ripoll \cite[Theorem 1.6]{casteras} obtained the following.

\begin{theorem}\label{CasterasMainTheorem} Let $N^{n+1}$, $n\ge 2$,
be a Cartan-Hadamard manifold satisfying the {\em SC} condition and let
$\Sigma \in \partial_{\infty} N$ satisfy $\Sigma = \partial A$ for some open subset $A \subset \partial_\infty N$ with $A = \overline{{\rm int}(A)}$. Then, there exists a closed set $W \subset N$ of locally finite perimeter in $N$ such that $M \doteq \partial [W]$ is a locally rectifiable, minimizing $n$-current in $N$, $\partial_\infty W = A$ and $\partial_\infty {\rm spt}M = \Sigma$.
\end{theorem}

\begin{remark}\label{regM}
An $n$-dimensional area minimizing $n$-rectifiable current $M$ in a smooth 
complete manifold $N^{n+1}$ is a smooth, embedded manifold on the complement of a singular set of 
Hausdorff dimension at most $n-7$. In particular, if $n<7$ then
the singular set is empty, while if $n=7$ it consists of isolated points. In higher codimensions, the singular set has Hausdorff dimension
at most $n-2$;
for more details see \cite[Section 3, Theorems 3.3, 3.4 and 3.5]{lellis}.
\end{remark}

\subsection{Geometry of the Ilmanen metric} Let us examine here the geometry of the Ilmanen
metric in more detail. As a matter of fact, we will compute its curvatures and geodesics. 
\begin{lemma}\label{ilmenenscurv}
The Riemannian manifold $N=(\mathbb{H}^{n+1}, \gI(n))$ is Cartan-Hadamard, and its sectional curvatures satisfy:
\begin{enumerate}[\rm(1)]
\item $\secI (\partial_i \wedge \partial_0)=-e^{-\frac{2}{nx_0}}\frac{2+n}{n x_0}$,
\smallskip
\item $ \secI (\partial_i \wedge \partial_j) =  -e^{-\frac{2}{nx_0}}{\frac{1+nx_0}{n}}$,
\smallskip
\item $\secI ((\sin\theta\, \partial_0+ \cos\theta\,\partial_i) \wedge \partial_j)= \sin^2\theta\,\secI(\partial_0 \wedge \partial_j)+ \cos ^2\theta\,\secI (\partial_i \wedge \partial_j),$ 
	\end{enumerate}
	 for any $i,j\in\{1,\dots,n\}$
and $\theta \in (0, 2\pi)$.
\end{lemma}

\begin{proof}
Completeness immediately follows from
$\gI(n) > \g_{\h},$
and the fact that $(\mathbb{H}^{n+1}, \g_{\h})$ is complete. Direct computations gives the sectional curvatures in (1), (2) and (3). Eventually, let
$\pi \subset T_p \mathbb{H}^{n+1}$ be a fixed $2$-plane. If $\pi \subset \partial_0^\perp$,
then up to a rotation, $\pi$ is generated by $\partial_i \wedge \partial_j$ and (b) gives $\secI(\pi) \le 0$. 
Otherwise, let $e_1$ be a unit vector generating $\pi \cap \partial_0^\perp$, which up to rotation 
we can assume to be $\partial_1$. Complete $e_1$ to an orthonormal basis $\{e_1,e_2\}$ of
$\pi$. Again up to rotation, we have that
$e_2 = \sin \theta \partial_0 + \cos \theta \partial_2,$
for some $\theta \in (0,2\pi)$. Item $(3)$ gives now
$\secI(\pi) \le 0.$
\end{proof}

Let $\gamma:(-\varepsilon,\varepsilon)\to\mathbb{H}^{n+1}$ be a geodesic with respect to the
Ilmanen metric $\gI(n)$.
Observe that isometries of the hyperbolic space
which preserves the direction $-\partial_0$ are also isometries of the Ilmanen metric.
 Because of this fact, it suffices to examine only geodesics in the $x_0x_1$-plane.
Let us suppose that the geodesic has the form $\gamma=(x_0,x_1)$.
By straightforward computations we see that $\gamma$ is a
geodesic if the functions $x_0$ and $x_1$ satisfy 
$$
x''_0 - \frac{1+nx_0}{nx_0^2} \big( (x'_0)^2
		- ( x'_1)^2\big)=0\quad\text{and}\quad
		x''_1 - 2\frac{1+nx_0}{nx_0^2}  x'_0 x'_1=0.
$$
Following the same methods as in Section \ref{sec3}, we can show the following:

\begin{lemma}\label{geoilmanen} Let $\gamma:\R\to (\mathbb{H}^{n+1},\gI(n))$ be a geodesic
lying in the $x_0x_1$-plane. Then, either $\gamma$ is a vertical line or:
	\begin{enumerate}[\rm(1)]
	  	\item there exist $t_n \in \real{}$ such that $\max x_0 = x_0(t_n)$, and $\gamma$ is symmetric with respect to the line $x_1=t_n$ in the $x_0x_1$-plane;
		\smallskip
	  	\item $\gamma$ is concave, when
		regarded as a curve lying the Euclidean space;
		\smallskip 
		\item the two ends of $\gamma$ approach $\bi \h^{n+1}$ orthogonally.	  	
	  \end{enumerate}
\end{lemma} 

\subsection{Proof of Theorem \ref{teoPlateau}}
According to Lemma \ref{ilmenenscurv}, the Ilmanen space $N = (\h^{n+1},\gI(n))$ is a
Cartan-Hadamard manifold. Regarding the SC condition, if $x \in \bi \h^{n+1}$ one can consider a 
spherical barrier $\SM$, which by Lemma \ref{sc-conv} is convex with respect to the 
upward pointing normal direction. The SC condition therefore holds at any point $x\in \bi \h^{n+1}$
by choosing as
$\Omega$ the half-ball below $\SM$. However, at the point ${\rm p}_\infty$, the SC condition may fail
and for  this reason we have to slightly complement the strategy in \cite{lang,casteras}, which we now recall. Fix an origin $o \in N$ and consider the cone $\mathrm{Cone}(o,A)$ generated by geodesics issuing 
from $o$ to points in $A$. Moreover, for each $i \in \mathbb{N}$, consider the set
$$T_i = \partial B_i(o) \cap \mathrm{Cone}(o,A)$$
with orientation pointing outside of $B_i(o)$ and denote by $[T_i]$ its associated $n$-rectifiable current.
Note that the boundary $\partial[T_i]$ is supported in
$\mathrm{Cone}(o,\Sigma)$.
Meanwhile, according to Lemma \ref{geoilmanen}, the geodesics in
$N$ either are vertical lines or behave like grim-reaper type curves. Since $A$ is relatively compact in $\bi \h^{n+1}$, we can therefore take a large enough bowl soliton $\mathscr{B}$ such that $\overline{\mathrm{Cone}(o,A)}$ (in particular, $\overline{T}_i$) lies in the open subgraph of
$\mathscr{B}$,  which we call $U$. According to a result of Lang \cite{lang}, for each $i\in\n$, there exists
a set $W_i \subset \overline{B_i(o)}$ of finite perimeter such that
$$M_i \doteq \partial[W_i] - [T_i],$$
is area minimizing in $\overline{B_i(o)}$. Notice that
$\partial M_i = - \partial [T_i]$ is supported in $U$. Moreover, since $B_i(o)$ is strictly convex,
by the strong maximum principle of White \cite{whi10} we deduce that
$$
\operatorname{spt} M_i \cap \partial B_i(o) = \operatorname{spt} \partial M_i,\qquad i\in\n.
$$
We claim now that ${\rm spt} M_i \subset U$. Suppose to the contrary that this is not true and consider
the foliation of $\h^{n+1}$  determined by $\mathscr{B}$. Then we could find a large bowl soliton
$\mathscr{B}'$ lying above $\mathscr{B}$
and touching $\operatorname{spt}M_i$ from above at some point
$p \not \in \operatorname{spt} \partial M_i$. Let $U'$ 
be the open set below $\mathscr{B}'$, and consider the manifold with boundary
$$N' = \overline{U'} \cap B_i(o).$$
Let $v(M'_i)$ be the stationary integral varifold obtained, by forgetting orientations, from the connected 
component of $M_i$ whose support contains $p$;
see \cite[Section 27]{simon}. The strong maximum principle of White \cite[Theorem 4]{whi10},
guarantees that
$\operatorname{spt} v(M_i') \cap N'$ contains a connected component of $\mathscr{B} \cap B_i(o)$. 
In particular, 
$\operatorname{spt} \partial M_i'$ contains a piece of $\mathscr{B} \cap \partial B_i(o)$. This however
contradicts
$ \partial M_i' \subset \operatorname{spt} \partial M_i \subset U.$ 
Having observed that each $W_i$ is contained in $U$ and is therefore separated from
${\rm p}_\infty$, the rest of the argument follows verbatim as in \cite{lang,casteras}.

\begin{remark} 
A similar argument (see \cite[Theorem 1.5]{casteras}) would allow to solve Plateau's problem for $k$-dimensional submanifolds 
provided that the point ${\rm p}_\infty$ can be separated from each $\partial [T_i]$ by a $k$-convex
barrier.  We have been unable to produce such objects, e.g. by rotating special curves 
around the $x_0$-axis. It might be possible that such barriers do not exist.  
\end{remark}

\section{The Dirichlet problem at infinity}\label{secdir}
In this section, we investigate the Dirichlet problem \eqref{QQ}.
Set for simplicity 
	\[
	f(u) =-\frac{1+mu}{u^2}\quad\text{and}\quad W(u)=\frac{1}{\sqrt{1+|Du|^2}}.
	\]
	and consider the operator $\QQ$ given by
\begin{equation*}
	\QQ[u] = \diver\left(\frac{D u}{W(u)}\right) - \frac{f(u)}{W(u)},
\end{equation*}
acting on positive $C^2$-functions on $\Omega$.

\subsection{Proof of Theorem \ref{teoDir}. Part (1) - Existence:}
We shall first solve the problem for data that do not meet
the boundary at infinity of the hyperbolic space, and then we will exploit Perron's method to establish the
existence of solutions to the Dirichlet problem at infinity.

We distinguish three cases:

{\bf Case A:} Assume at first that $\partial\Omega$ is compact, $\phi$ is positive on $\partial \Omega$ and $C^3$-smooth.
To solve the problem use the continuity method. Since the operator $\QQ$ is quasilinear,
the method will be applicable
once we provide global a priori $C^1$-estimates on a solution $u$ of \eqref{QQ};
see for example \cite[Chapter 11]{gilbarg} or \cite[page 417]{serrin2}.
Height and gradient 
estimates will be proved
by constructing suitable subsolutions and supersolutions for \eqref{QQ} and applying 
classical comparison  theorems, for which we refer to  \cite[Theorem 2.1.3 \& 2.1.4]{pucci}. Notice that comparison holds because $f$ defined above is increasing. 

{\bf Claim 1}\! ({\em Height estimate}){\bf:} {\em There exist positive constants $B_1$
and $B_2$ which only depend on $\phi$ and $\Omega$ such that $B_1\le u\le B_2$.}

{\em Proof of the claim.} Observe at first that the constant
$u_1 =B_1= \min_{\partial \Omega} \phi$ is a subsolution to \eqref{QQ}, namely
$\QQ[u_1] \ge 0$.
By the comparison principle, a solution $u$ to \eqref{QQ} satisfies $u_1\le u$. Consider now
a bowl soliton $\mathscr{B}$ lying above the graph of $\phi$ on $\partial \Omega$. Then, if
$\mathscr{B}$ is generated by the graph of $u_2$, by the comparison principle we get $u\le u_2$. The thesis follows.
 \hfill{$\circledast$}

{\bf Claim 2}\! ({\em  Boundary gradient estimate}){\bf:} {\em There exists a constant
$B_3$ which only depends on $\phi$ and $\Omega$ such that
$\sup_{\partial\Omega}|Du|\le B_3.$
}

{\em Proof of the claim.} Let $\nu$ be the inward pointing Euclidean unit normal on
$\partial \Omega$, and fix a tubular neighborhood
$$
\Omega_{\rho} = \{ x \in \Omega : \operatorname{dist}(x, \partial \Omega) < \rho\}
$$
around $\partial\Omega$ for which the Fermi chart
$[0,\rho) \times \partial \Omega \to \Omega_\rho$ given by $(r,y) \mapsto y+r \nu(y)$ 
is well defined and smooth. Define $\hat\phi$ on $\Omega_{\rho}$ by $\hat\phi(r,y)=\phi(y)$. The goal is to find $l\in(0,\rho)$ and an increasing $C^2$-smooth function $\psi:[0,l)\to [0,\infty)$ with bounded gradient and satisfying
$$
\psi(0)=0\quad\text{and}\quad |u(r,y)-\phi(y)| \le \psi(r) \qquad \forall\, (r,y)\in\Omega_{l}.
$$
To achieve this goal we seek for $l\in(0,\rho)$ and $\psi$ such that
$\QQ[\psi+\hat\phi]\le 0 \le \QQ[-\psi + \hat \phi]$  on $\Omega_{l}$. Set $v=\psi+\hat\phi$. Then, by a straighforward computation we get that
\begin{eqnarray}\label{QQformula}
\QQ[v] &=& \dfrac{1}{W(v)}\left[\Delta v-D^2v \left(\dfrac{D v}{W(v)} ,\dfrac{D v}{W(v)}\right)-f(v)\right]\nonumber\\
&=&\dfrac{1}{W(v)}\left[\psi''+\psi'\Delta r +\Delta \hat\phi-
\psi''\frac{\langle D r,D v\rangle^2}{W^2(v)}-\psi'D^2r
\left(\dfrac{D v}{W(v)}, \dfrac{D v}{W(v)}\right) \right.\nonumber\\
& & \left. \qquad\quad\,-D^2 \hat\phi\left(\dfrac{D v}{W(v)} ,\dfrac{D v}{W(v)}\right)-f(v) \right],
\end{eqnarray}
where here $\Delta$ is the Euclidean Laplacian. Let us examine each term of 
\eqref{QQformula} carefully:
\begin{itemize}
\item
Because $f$ is increasing and negative, we deduce that
\begin{equation}\label{def_cphi}
0<-f(t) \le C_\phi\doteq -f(B_1) \quad \text{for each}
\quad t \ge {\inf}_{\partial \Omega} \phi>0.
\end{equation}
\item
Since $\phi$ is assumed to be smooth, there exists a constant $C_1$ such that	
\begin{equation}\label{C1constant}
1+ \|D \hat\phi\|^2_\infty + \|D^2 \hat\phi\|^2_\infty \le C_1 \quad\text{on}\quad \Omega_{\rho}.
\end{equation}
\item
Using Gauss' Lemma we see that
\begin{equation}\label{Gauss}
D v= \psi'D r+D \hat\phi\quad\text{and}\quad \langle D\hat\phi,Dr\rangle=0.
\end{equation}
\item
Using the fact that
$\mathcal{Y}=Dv/W(v)$
is of length at most $1$, we deduce that there exists a constant $C_2$ such that
\begin{equation}\label{C2constant}
\langle D r, \mathcal{Y}\rangle = \frac{\psi'}{W} \quad\text{and}\quad
\big|\Delta \hat\phi\big| + \big|D^2 \hat\phi\left(\mathcal{Y}, \mathcal{Y}\right)\big| \leq C_2.
\end{equation}
\item
Since $Dr$ belongs to the kernel of $D^2r$, we have that
\begin{equation}\label{C3constant}
\big| D^2r(\mathcal{Y}, \mathcal{Y})\big|=\big|D^2r( D\hat\phi, D\hat\phi)\big| \leq
C_1^2 \|D^2r\|_{L^\infty(\Omega_\rho)}\doteq C_3.
\end{equation}
\item
Since $H_{\partial\Omega}\ge 0$, the Laplacian comparison theorem implies 
\begin{equation}\label{eq_ine}
	\Delta r \le -\frac{H_{\partial\Omega}}{1-rH_{\partial\Omega}} \le 0 \quad \text{on } \, \Omega_\rho.
	\end{equation}
\end{itemize}
Taking into account \eqref{def_cphi}, \eqref{C1constant}, \eqref{Gauss},
\eqref{C2constant}, \eqref{C3constant}, \eqref{eq_ine},  inequality \eqref{QQformula}
can be estimated by
\begin{equation}\label{eq_boundsopraQ}
W^3(v) \QQ[v] 
\le W^2(v)\psi'' -\psi''(\psi')^2+C_3\psi'+ CW^2(v), 
\end{equation}
where $C = C_2 + C_\phi.$
Consider now $\psi$ to be a function of the form $\psi(r)=\mu\log(1+kr),$
where $\mu>0$ and $k>0$ are constants to be chosen later. Then, 
$$
\psi(0) = 0, \qquad \psi'(r)=\dfrac{\mu k}{1+kr} \quad\text{and}\quad  \psi''(r)=-\dfrac{(\psi')^2}{\mu}.
$$
For this choice of $\psi$, inequality \eqref{eq_boundsopraQ} becomes
\begin{equation}\label{eq_boundsopraQ1}
W^3(v) \QQ[v]\le-\dfrac{(\psi')^2}{\mu}\big(1+|D\hat\phi|^2\big)+C_3\psi'+CW(v)^2.
\end{equation}
Since
$$W^2(\psi)= 1 + |D\hat \phi|^2 +(\psi')^2 \le C_1 + (\psi')^2$$
we get
\begin{equation}\label{eq_impo_grad}
W^3(v)\QQ[v]\leq \Big(-\dfrac{1}{\mu} +C\Big)(\psi')^2+C_3\psi'+ CC_1.
\end{equation}
By height estimates, $u \le B_2$ on $\Omega$. Define 
\[
\mu=\dfrac{B_2}{\log(1+ \sqrt{k})},
\]
and observe that $\mu\rightarrow 0$ as $k\rightarrow \infty$. Choose $k > \rho^{-2}$
sufficiently large and set $l_1=k^{-1/2}$. Then, from \eqref{eq_impo_grad} we deduce
that $\QQ[v] < 0$ on
$\Omega_{l_1}$. Since
$$v(0,y)=\phi(y)=u(0,y)\quad\text{and}\quad v(l_1,y)> B_2\ge u(l_1,y) \quad\text{for each} \quad
y\in\partial{\Omega},$$
from the comparison principle, it follows that $u\le v=\psi+\phi$ on $\Omega_{l_1}$.
To prove the existence of $l_2\in(0,\rho)$ such that
$u\ge -\psi+\phi$ on $\Omega_{l_2}$, we may proceed with the same technique as above and making use the fact that $-f(t)\ge -f(B_2)$ for $t\in[B_1,B_2].$ Alternatively, observe that since $f \le 0$ a standard lower barrier $w$ for the minimal surface equation on $\Omega_\rho$ also satisfies $\QQ[w] \ge 0$ on $\{w > 0\}$. Take $l=\min\{l_1,l_2\}$ and let us restrict ourselves in $\Omega_{l}$.
For each $y\in\partial\Omega$, we have that
$$
\left|\frac{\partial u}{\partial\nu}(y)\right|
=\lim_{r\to 0}\frac{|u(r,y)-u(0,y)|}{r}\le \lim_{r\to 0}\frac{\psi(r)}{r} = \mu k.
$$
Let now $X$ is a unit tangent vector field along $\partial\Omega$. Since
$$u|_{\partial\Omega}\equiv\phi|_{\partial\Omega},$$
the derivative of $u$ in the direction $X$ obeys
$$
\big|\langle Du,X\rangle\big|\le\sup_{\partial\Omega}|D\phi|.
$$
Combining all these we complete the proof of the claim. \hfill{$\circledast$}

{\bf Claim 3}\! ({\em  Interior gradient estimate}){\bf:}
{\em There exists a constant
$B_4$ which depends only on $\phi$ and $\Omega$ such that
$\sup_{\Omega}|Du|\le B_4.$
}

{\em Proof of the claim:} From Lemma \ref{graph1}, the unit normal $\nu$ and the mean
curvature of the soliton $M$ are given by
$$
\nu=\frac{u\,\partial_0-uD u}{\sqrt{1+|D u|^2}}\quad\text{and}\quad H=\frac{-1}{u\sqrt{1+|Du|^2}}<0.
$$
Let us compute the Laplacian of $H$ with respect to the induced metric $\g$,
following the same lines as in \cite[Lemma 2.1]{fra14}.
For simplicity, let us denote the metrics of $\h^{n+1}$ and $M$
by the same letter $\g$. As usual let us denote by $\overline{\nabla}$ the Levi-Civita
connection of $\h^{n+1}$ and by $\nabla$ the Levi-Civita connection of the induced
metric on $M$.
Let $\{e_1,\dots,e_n\}$ be a local orthonormal 
tangent frame, which is normal at a fixed point $p\in M$, and denote by $b_{ij}$ the coefficients
of  the second fundamental form $\II$ of $M$ with respect to $\nu$. Differentiating with respect to $e_i$, we get that
$$
e_iH=-e_i\g_{\h}(\partial_0,\nu)=
-\g_{\h}(\overline{\nabla}_{e_i}\partial_0,\nu)-\g_{\h}(\partial_0,\overline{\nabla}_{e_i}\nu),
\qquad i\in\{1,\dots,n\}.
$$
From  the Koszul formula, we have that at $p$ it holds
$$\overline{\nabla}_{e_i}\partial_0=-u^{-1}e_i, \qquad i\in\{1,\dots,n\}.$$
Consequently,
\begin{eqnarray*}
e_iH=-\g_{\h}(\partial_0^{\top},\overline{\nabla}_{e_i}\nu)=\II(\partial_0^{\top},e_i),
\qquad i\in\{1,\dots,n\}.
\end{eqnarray*}
Differentiating once more, using Codazzi, and then estimating at $p\in M$, we obtain
\begin{eqnarray*}
\Delta_{\g} H&=&e_ie_iH
=e_i\big(b_{ij}\g_{\h}(\partial_0,e_j)\big)
= b_{iji}\g_{\h}(\partial_0,e_j)+ b_{ij}\g_{\h}(\overline{\nabla}_{e_i}\partial_0,e_j)+
b_{ij}\g_{\h}(\partial_0,\overline{\nabla}_{e_i}e_j)\\
&=&\g_{\h}(\partial_0^{\top},b_{iij}e_j)-u^{-1}b_{ij}\delta_{ij}+
b_{ij}b_{ij}\g_{\h}(\partial_0,\nu)\\
&=&\g_{\h}(\partial_0^{\top},\nabla H)- u^{-1}H-H|\II|^2,
\end{eqnarray*}
where $\Delta_{\g}$ is the Laplacian with respect to the induced metric $\g$ of the
soliton.
Since $-H>0$, according to the maximum principle, we obtain that
$\sup_{\Omega}H=\max_{\partial\Omega}H.$
Thus, there exists $y_0\in\partial\Omega$ such that
$$
\frac{-1}{u(x)\sqrt{1+|Du|^2(x)}}\le\frac{-1}{u(y_0)\sqrt{1+|Du|^2(y_0)}}
\qquad\text{for each }\, x\in\Omega.
$$
Hence,
$$
|Du|^2(x)\le\frac{u^2(y_0)}{u^2(x)}\Big(1+|Du|^2(y_0)\Big)-1 \qquad \text{for each }\, x\in\Omega.
$$
Combining with the estimates we showed in Claims 1 and 2,
we deduce the desired estimate on the gradient of $u$. This completes the proof of the claim.
\hfill{$\circledast$}

\medskip

{\bf Case B:} 
Assume that $\Omega$ has $C^3$-smooth compact mean convex boundary $\partial\Omega$.
Furthermore, assume that $\phi : \partial \Omega \to (0,\infty)$ is 
continuous. Choose a decreasing sequence $\{\phi_j\}$ and
an increasing sequence $\{\theta_j\}$ of positive smooth functions uniformly converging to $\phi$. For each $j\in\n$ denote by $u_j, v_j$ the solutions given
by Case A for boundary data $\phi_j$ and
$\theta_j$, respectively. By the comparison maximum principle, the sequence $\{u_j\}$ is decreasing,
$\{v_j\}$ is increasing and $v_j \le u_l$ for each $j,l\in\n$. Moreover, by the height estimates obtained
in Case A, there exist positive constants $B_1,B_2$ such that
\[
B_1 \le v_j \le u_l \le B_2 \qquad \text{for all }\, j,l\in\n.
\]
According to a result of Simon \cite[Corollary 1, p. 257]{simon2}, the sequences $\{u_j\}$ and $\{v_j\}$ have uniformly bounded gradients on compact subsets $K \subset \Omega$. Then local $C^{1,\alpha}$-estimates follow by Ladyzhenskaya \& Ural'tseva \cite{Lady}, and Schauder estimates imply that the sequences $\{u_j\}$ and  $\{v_j\}$ are bounded on $C^{2,\alpha}(K)$; for more details see
also \cite[Section 11.3, Chapter 13 and Theorem 13.6]{gilbarg}. Passing to the limit, using the same
idea as in Lemma \ref{lemmaxgraph}, we deduce that  $u_j \downarrow u$ and $v_j \uparrow u$ 
locally in  $C^{2,\alpha}$ to a unique solution $u$ to $\QQ[u] = 0$ which satisfies $u \equiv \phi$
on $\partial\Omega$.

\medskip

{\bf Case C:} We conclude the proof by considering the case where $\Omega$, not necessarily compact but satisfying $H_{\partial \Omega} \ge 0$, is 
contained between two parallel hyperplanes of
$\partial_\infty'\h^{n+1}$. We employ Perron's method, the main novelty being the treatment of boundary barriers to force $u = \phi$ on $\partial \Omega$. Recall at first that a function $v$ is said to 
satisfy
$\QQ[v] \ge 0$ in the viscosity sense if, for each $x \in \Omega$ and each $C^2$-smooth test function
$\varphi$ touching $v$ from above at the point $x$ (i.e., $\varphi \ge v$ near $x$ and $\varphi(x)= v(x)$) it 
holds $\QQ[\varphi] \ge 0$; see for details \cite{cil}.
Consider now a large grim-reaper cylinder $\mathscr{G}$ such that the graph of $\phi$ over $\partial \Omega$ lies in the region below $\mathscr{G}$. Without loss of generality, we may denote the graph function generating $\mathscr{G}$ with the same name.  Define Perron's class
\[
\mathscr{F} = \left\{ v \in C(\overline\Omega) \ : \begin{array}{ll}
0 < v \le \mathscr{G} & \text{ on } \,\, \Omega, \\[0.3cm]
\QQ[v] \ge 0 & \text{ on }\,\, \Omega\, \text{ in the viscosity sense,} \\[0.3cm]
0 \le v \le \phi & \text{ on } \, \partial \Omega.
\end{array}\right\}.
\]

{\bf Claim 4:} {\em The set $\mathscr{F}$ is non-empty.}

{\em Proof of the claim:} For each $x \in \Omega$ we can consider the maximal spherical barrier $\mathcal{S}_x$
centered at $x$ whose boundary at infinity is contained in $\overline{\Omega}$. The spherical
barrier can be expressed as the graph of a function $s_x:B_x\to\R$ with solves $\QQ[s_x] \ge 0$ in the interior of its domain of definition $B_x$.
Note that $s_x$ is zero on $\partial B_x$ and 
$s_x < \mathscr{G}$ on $B_x$ by comparison. Extend $s_x$ as being zero on $\Omega \backslash B_x$ and define $s : \Omega \to \R$ by
$s = \sup_{x\in\Omega} s_x.$
Then, $0 < s \le \mathscr{G}$ on $\Omega$ and moreover $s$ is
locally Lipschitz. By elementary properties of viscosity
solutions, we deduce that $\QQ[s] \ge 0$ on the entire $\Omega$. Since $\Omega$ is contained between two parallel hyperplanes, the radius of $\mathcal{S}_x$ is uniformly 
bounded from above by some $R>0$. Hence, for each $x \in \Omega$, by considering a nearest point
$x_0 \in \partial \Omega$ to $x$ and a spherical cap of radius $R$ and center $x_0 + R \nu(x_0)$ we deduce
\[
s_y(x) \le R \sqrt{ 1 - \frac{R - \mathrm{dist}(x, \partial \Omega)}{R}} \qquad \forall \, y \in \Omega.
\]
Consequently, $s \in C(\overline{\Omega})$ and $s \equiv 0$ on $\partial \Omega$. This shows that
$\mathscr{F} \neq \emptyset$. \hfill{$\circledast$}

Define now Perron's envelope
\begin{equation}\label{Pen}
u(x) = \sup \left\{ v(x) \ : \ v \in \mathscr{F} \right\}.
\end{equation}
Then, $u$ is lower-semicontinuous on $\overline{\Omega}$, $0< u \le \mathscr{G}$ on $\Omega$ and $0 \le u \le \phi$ on $\partial \Omega$.

{\bf Claim 5:} {\em The function $u$ defined in \eqref{Pen} belongs to $C^\infty(\Omega)$ and
$\QQ[u]=0$ on $\Omega$}.

{\em Proof of the claim:}
Fix $x \in \Omega$ and a sequence $\{v_j\} \subset \mathscr{F}$ with $v_j(x) \to u(x)$. Up to replacing $v_j$ with $\max\{v_1, \ldots, v_j\} \in \mathscr{F}$, we can assume that $v_j(x) \uparrow u(x)$. Pick a small ball $B \subset \Omega$ centered at $x$, and for each $j\in\n$ solve
$$
\left\{ 
\begin{array}{ll}
\QQ[v'_j] = 0 & \quad \text{on } \, B, \\[0.2cm]
v'_j = v_j  & \quad \text{on } \, \partial B.
\end{array}
\right.
$$
The existence of the unique $v_j' \in C^2(B) \cap C(\overline{B})$ follows by Case B above. From the
comparison principle we deduce that $v_j \le v_j' \le \mathscr{G}$ on $B$,
for each $j\in\n$; see \cite{cil}\footnote{Comparison in this case holds trivially: if $\max_B (v_j-v_j') = c > 0$, the function $v_j' + c$ would touch from above $v_j$ at come interior point $x_0$. However, $\QQ[v_j'+c]<0$, contradicting the fact that $v_j$ is a subsolution at $x_0$.}.

Define the replacement $\tilde{v}_j$ of $v_j$ to be the function
$$
\tilde{v}_j=\left\{ 
\begin{array}{ll}
v_j' & \quad \text{on } \, B, \\[0.2cm]
v_j  & \quad \text{on } \, \Omega\backslash B.
\end{array}
\right.
$$
Then $\tilde{v}_j$ still belongs to $\mathscr{F}$ and $\tilde v_j(x) \to u(x)$. Local gradient estimates 
\cite[Corollary 1]{simon2} 
and higher elliptic regularity imply $\tilde v_j \to v \le u$ locally smoothly on $B$, where $v$ is a function with
$v(x) = u(x)$. We claim that $u \equiv v$ on $B$. Assume to the contrary that $u(p) > v(p)$ for some $p \in B$, let 
$w \in \mathscr{F}$ such that $w(p) > v(p)$ and consider $w_j = \max\{\tilde v_j, w\}$. Let $\tilde{w}_j$ be 
the replacement of $w_j$ on $B$. Again elliptic estimates guarantee that
$\tilde{w}_j \to \tilde{w}\le \mathscr{G}$ on $B$ locally smoothly. By construction, $\tilde w \ge v$ on $B$, with strict inequality at $p$ 
but with equality at $x$, contradicting the maximum principle. \hfill{$\circledast$}

{\bf Claim 6:} {\em The function $u$ defined in \eqref{Pen} is continuous up to the
boundary $\partial\Omega$ and $u\equiv\phi$ on $\partial \Omega$.}

{\em Proof of the claim:} Fix a point $x_0 \in \partial \Omega$, choose a positive
$\eps>0$ and a large
ball $B_{r_0}$ centered at some fixed origin for which 
$x_0 \in B_{r_0-2}$. To simplify the notation, let us denote here the intersection of $\partial\Omega$
with a ball $B_r$ of radius $r$ by
$\partial\Omega_{r}$, that is
$\partial\Omega_r=\partial \Omega \cap B_{r}.$
Parametrize now the closure of the portion $\Omega\cap B_{r_0}$ by the smooth 
Fermi chart $[0,\rho) \times \partial\Omega_{r_0} \to \overline{\Omega}$ given by
$$(r,y) \mapsto y+r \nu(y),$$
where $\nu$ the unit normal to the boundary $\partial\Omega$ pointing towards $\Omega$. Let
$B \ge \|\mathscr{G}\|_\infty$ and consider functions $\phi_1, \phi_2 \in C^3(\partial \Omega)$
satisfying the following properties:
\[
\begin{array}{cll}
(i_1) & \quad 0 \le \phi_2 \le \phi\quad \text{and}\quad \eps + \phi \le \phi_1 \le \mathscr{G}, & \text{on } \,
\partial \Omega, \\[0.2cm]
(i_2) & \quad |\phi_j - \phi| \le 2\eps, & \text{on } \, \partial\Omega_{r_0-2}, \\[0.2cm]
(i_3) & \quad \phi_2 = 0 \quad\text{and}\quad \phi_1 = \mathscr{G}, & \text{on } \,
\partial \Omega \cap (\bi\h^{n+1} \backslash B_{r_0-1}).
\end{array}
\]
By the construction of boundary gradient estimates in Case 1, there exists $\rho_0 < \rho$
(depending on $\eps$) and $C^2$-smooth functions $v_1$ and $v_2$ on
$U = [0,\rho_0] \times \partial \Omega_{r_0}$
with the following properties:
\begin{itemize}
\item $\QQ[v_1] \le 0$ on $U$, $\QQ[v_2] \ge 0$ on $\{v_2 > 0\}$.
\medskip
\item
It holds
$$v_1(\rho_0,y) = B,\quad  v_1(0,y) = \phi_1(y),\quad \partial_r v_1(r,y) > 0,
\quad\text{for each}\,\, (r,y) \in U,$$
and
$$v_1(r,y) \ge \mathscr{G}(r,y),\quad\text{for each}\,\, (r,y) \in [0,\rho_0) \times (\partial\Omega_{r_0} \backslash\partial\Omega_{r_0-1}).
$$
\item It holds
$$v_2(\rho',y) < 0, \quad v_2(0,y) = \phi_2\quad\text{and}\quad\partial_r v_2(r,y) < 0, \quad
\text{for each}\,\,(r,y) \in U.$$
Therefore,
$$v_2(r,y) <0\quad\text{on}\quad (0,\rho_0) \times (\partial\Omega_{r_0} \backslash \partial\Omega_{r_0-1}).$$
\end{itemize} 
Pick now a smooth function $\eta : \partial\Omega_{r_0} \to [0,\rho_0]$ satisfying
$$\eta(y) = \rho_0\quad\text{for}\quad y \in \partial\Omega_{r_0-1}\quad\text{and}\quad
\eta(y) = 0\quad\text{ for}\quad y \in \partial \Omega_{r_0} \cap \partial B_{r_0},$$
and consider the region 
\[
\widetilde{U} = \big\{(r,y) \in U \ : \ 0 \le r < \eta(y)\} \subset U. 
\]
Then, by construction,
$$v_2 < 0 < \mathscr{G} \le v_1\quad \text{on} \quad\partial \widetilde{U}\backslash \partial \Omega.$$
By the comparison principle, we have that $v \le v_1$ on $\widetilde{U}$ for each $v \in \mathscr{F}$. Thus $u \le v_1$ on $\widetilde{U}$ and 
\[
\limsup_{x \to x_0} u(x) \le \limsup_{x \to x_0} v_1(x) = \phi_1(x_0) \le \phi(x_0) + 2\eps. 
\]
On the other hand, by construction we have that $\{v_2 > 0\} \subset \widetilde{U}$.
Therefore, the function $v$ given by
$$
v=\left\{ 
\begin{array}{ll}
\max\{v_2,s\} & \quad \text{on } \, \widetilde{U}, \\[0.2cm]
s  & \quad \text{elsewhere on } \, \Omega,
\end{array}
\right.
$$
is well defined on the entire $\Omega$ and $v \in \mathscr{F}$. Inequality $u \ge v$ implies 
\[
\liminf_{x \to x_0} u(x) \ge \lim_{x \to x_0} v(x) = \lim_{x \to x_0} v_2(x) = \phi_2(x_0) > \phi(x_0)-2\eps.
\]
The continuity of $u$ at $x_0$ follows by letting $\eps \to 0$. \hfill{$\circledast$}

This conclude the proof of Theorem \ref{teoDir}(1).

\subsection{Proof of Theorem \ref{teoDir}. Part (2) - Non-existence:} Suppose that there exists a point
$y\in\partial\Omega$ with $H_{\partial\Omega}(y)<0$, and let $u\in C^{\infty}(\Omega)\cap C(\overline\Omega)$ be a positive solution to $\QQ[u]=0$ on $\Omega$.
Our approach follows \cite{serrin2}, and we split the argument into three steps. The main difference with \cite{serrin2} is Lemma \ref{lema1} below: as observed in Remark \ref{rem_diri_intro}(3), its use to construct a boundary data $\phi$ for which the Dirichlet problems is not solvable forces a lower bound on the oscillation of $\phi$. 

To achieve our goal, we need to compare $u$ with appropriate supersolutions to \eqref{SE}. Recall that a function $r$ defined on an open subset of $\Omega$ is called a distance function if it smooth and $|Dr|\equiv 1$. We start with the following:

\begin{lemma}\label{lem_ODE}
Let $r$ be a distance function defined on an open subset $U\subset \R^n$
and $\omega:(0,\infty)\to\R$ a smooth function. Then $v=\omega(r):U\to\R$
is a supersolution of \eqref{SE}, if there exists a continuous function $h$ defined on $(0,r)$ such that
\begin{equation}\label{eq_nice_omega}
\omega' < 0,\quad \dfrac{\omega''}{\omega'[1+(\omega')^2]} + h \ge 0 \quad\text{and} \quad \Delta r - \frac{f(v)}{\omega'(r)} \ge h(r) \ge 0 \quad \text{on } \, U. 
\end{equation}
\end{lemma}

\begin{proof} Consider the orthonormal frame $\{e_1=Dr;e_2,\dots,e_n\}$. By
a straightforward computation, we deduce that
\begin{equation*}\label{eqsolitonH}
\sqrt{1+ (\omega'(r))^2}\, \QQ[v] = \disp \frac{\omega''(r)}{1+(\omega'(r))^2} +
\omega'(r)\Delta r -f(v)
\le \disp \frac{\omega''(r)}{1+(\omega'(r))^2} + h(r)\omega'(r),
\end{equation*}
from where the statement follows. 
\end{proof}

\begin{lemma}\label{lema1}
For each positive number $\varepsilon >0$, there exist positive constants
$a_0=a_0(\varepsilon, \operatorname{diam}(\Omega),n)$ and
$c =c(\operatorname{diam}(\Omega),n)$ such that
$$
{\sup}_{\Omega \backslash B_a(y)} u \leq \varepsilon+{\sup}_{\partial\Omega \backslash B_a(y)} u
-cf({\sup}_{\partial\Omega \backslash B_a(y) }u)
$$ 
for all $a< a_0.$
\end{lemma}

\begin{proof}
Let $a$ be a positive number. To simplify the notation let us set
$$d=2\operatorname{diam}(\Omega),\,\,\,\, U_a=\Omega\backslash B_a(y),\,\,\,\,
V_a=\partial\Omega\backslash B_a(y)\,\,\,\,\text{and}\,\,\,\, u_a^* = {\sup}_{V_a} u.$$
Denote by $r$ the distance function $r(x)=|x-y|$ and let $v:U_a\to\R$ be the function given by
$$
v=\omega(r)+ u_a^*,
$$  
where $\omega$ is $C^2$-smooth on $(a,d]$ and continuous on $[a,d]$.
Furthermore, we require for $\omega$ that
\begin{equation}\label{eq_ipo_omega}
\omega \ge 0, \quad \omega(d)=0, \quad \omega' < \frac{2d}{n-1} f(u_a^*) < 0 \quad\text{and}
\quad \omega'(a^+)=-\infty.
\end{equation} 
From the monotonicity of $f$ and Lemma \ref{lem_ODE}, we easily see that
\[
\Delta r - \frac{f(v)}{\omega'(r)} \ge \frac{n-1}{r} - \frac{f (u_a^*)}{\omega'(r)} \ge \frac{n-1}{r} - \frac{n-1}{2d} \ge \frac{n-1}{2r}>0.  
\] 
Therefore, $\QQ[v] \le 0$ provided that
\begin{equation}\label{eq_omega}
\dfrac{\omega''}{\omega'[1+(\omega')^2]} + \frac{n-1}{2r} \ge 0.
\end{equation}
We first find a solution $\tilde \omega$ to \eqref{eq_omega} with the equality sign. To achieve
this goal, consider the strictly decreasing diffeomorphism $F:(0,\infty)\to (0, \infty)$ given by
\begin{equation}\label{def_F_plat}
F(s)=\int^{\infty}_s\dfrac{\di\tau}{\tau(1+\tau^2)} = \log \sqrt{1 + s^{-2}}.
\end{equation}
By a direct computation we see that
$$
\left(F(-\tilde{\omega}')\right)'
=-\dfrac{\tilde{\omega}''}{\tilde{\omega}'[1+(\tilde{\omega}')^2]} = \dfrac{n-1}{2r}. 
$$
Integrating on $[a,r]$ and using the fact $\tilde{\omega}'(a^+)=-\infty$, we get
$$
-\tilde{\omega}'(r)=F^{-1}\left(\dfrac{n-1}{2}\log\left(\dfrac{r}{a}\right)\right).
$$
Another integration on $[r, d]$ gives 
\begin{equation}\label{defbaromega}
\tilde{\omega}(r)=\int_r^{d}F^{-1}\left(\dfrac{n-1}{2}\log\left(\dfrac{t}{a}\right)\right) \di t.
\end{equation}
Since $F^{-1}(t)\asymp t^{-1/2}$ as $t \to 0$, it follows that $\tilde{\omega}'$ is integrable in a neighbourhood of $a$. Also, explicit computation gives $\tilde \omega(a) \to 0$ as $a \to 0$. We can therefore choose $a_0=a_0(\varepsilon,d,n)$ small
enough so that
$\tilde\omega(a) < \varepsilon$ for each $a< a_0$. Summarizing, $\tilde{\omega}$
given in \eqref{defbaromega}
solves \eqref{eq_omega}. Choose now the function
$$
\omega(r)=\tilde{\omega}(r)-\dfrac{2d}{n-1}f(u_a^*)(d-r).
$$
Observe that $\omega$ satisfies both \eqref{eq_omega} and \eqref{eq_ipo_omega}. Hence,
$v=\omega(r)+ u_a^*$ is a supersolution to \eqref{SE}. We claim that $u \le v$ on the closure of
$U_a$. Indeed, assume by contradiction that $u - v$ has a positive 
maximum at some point $x_0$. By the strong maximum principle, $x_0$ is not an interior point, 
and thus
$x_0 \in \partial B_a(y)$ by the construction of $v$. However, along the segment $\ell$ given by
$$\ell(t) = x_0 + tDr(x_0)$$
it holds 
\[
(u\circ\ell-v\circ\ell)'(t) = \langle Du,Dr \rangle(\ell(t)) - \omega'(a+t) \to \infty \quad \text{as } \,\, t \to 0^+, 
\]
contradiction. From the inequality $u \le v$ on ${\overline{U}}_{a}$ we deduce
\[
u(x) \le u_a^* + \tilde{\omega}(a)-\dfrac{2d^2}{n-1}f(u_a^*) < u_a^* + \varepsilon -\dfrac{2d^2}{n-1}f(u_a^*) \quad\text{for all}\,\,\, a<a_0,
\]
concluding the proof of the lemma.
\end{proof}

\begin{lemma}\label{lema2}
For each $\varepsilon > 0$, there exists a positive constant $a_0=a_0(\Omega, \varepsilon)$
such that
\begin{equation}\label{eqsecondine}
{\sup}_{\Omega\cap B_a(y)}u\leq \varepsilon + {\sup}_{\Omega \cap \partial B_a(y)} u
\end{equation}
for all $a < a_0$.
\end{lemma}

\begin{proof} Define $r(x)=\dist(x,\partial\Omega)$, 
and choose $a_0$ small enough to guarantee that $r$ is smooth on
$\overline{\Omega} \cap B_{a_0}(y)$. Since $\Delta r(y)=-H_{\partial\Omega}(y)$, by continuity
there exist $a_0,\theta>0$ such that
$$
\Delta r\ge 2\theta\qquad \text{on }\, \Omega\cap B_a(y).
$$
Fix $a<a_0$, choose $\delta\in(0,a)$ and set
$$u_a^{**} = {\sup}_{\Omega \cap \partial B_a(y)} u,\qquad k =\theta^{-1}{\sup}_\Omega |f(u)|.$$
Consider now the function $v:\Omega\cap B_{a}(y)\to\R$
given by
$$
v =u_a^{**} +\omega(r),
$$ 
where $\omega$ is a $C^2$-smooth function on $(\delta,a]$ and continuous on
$[\delta, a]$. Furthermore, we require $\omega$ to satisfy 
\begin{equation}\label{eq_ipo_omega_2}
\omega > 0 \ \ \ \text{ on } \, [\delta, a], \qquad \omega'\leq -k \ \ \ \text{ on } \, (\delta, a]\qquad\text{and}
\qquad \omega'(\delta^+) = -\infty. 
\end{equation}
Observe that
\[
\Delta r - \frac{f(v)}{\omega'(r)} \ge 2\theta - \frac{\sup_\Omega |f(u)|}{k} \ge \theta \qquad \text{on } \, \Omega \cap B_a(y).
\]
By Lemma \ref{lem_ODE}, we deduce that $\QQ[v] \le 0$ provided 
\[
\dfrac{\omega''}{\omega'[1+(\omega')^2]} + \theta = 0
\]
and the conditions \eqref{eq_ipo_omega_2} are satisfied.
Consider the decreasing diffeomorphism $F : (0,\infty) \to (0,\infty)$ given in \eqref{def_F_plat} to 
rewrite the last ODE in the form
$$(F(-\omega'))' = \theta.$$
Integrating on $(\delta, r)$ and using that $\omega'(\delta^+)=-\infty$ we get 
\[
-\omega'(r)= F^{-1}\left(\theta(r-\delta)\right) \ge F^{-1}(\theta a_0) \ge k,  
\]
where the last inequality holds if $a_0$ is small enough. Moreover, the last requirement in
\eqref{eq_ipo_omega_2} is also satisfied. Integrating again and using the asymptotic behaviour of $F^{-1}$,
the function $\omega'$ is integrable in a neighbourhood of $\delta$ and thus 
$$
\omega(r) =\displaystyle \int_{r}^{a_{0}}F^{-1}(\theta(t-\delta))\di t \in C([\delta, a]) \cap C^2((\delta,a]).
$$
Consequently, all of the required assumptions on $\omega$ are satisfied. Making use of the
comparison maximum principle
as in the previous lemma, we obtain $u \le v$ on $\Omega \cap B_a(y)\backslash B_\delta(y)$.
Therefore, 
\[
u \le u_a^{**} + \omega(\delta) = u_a^{**} + \int_{\delta}^{a_{0}}F^{-1}(\theta(t-\delta))\di t.
\]
Changing variables from $s$ to  $t-\delta$ in the last integral, letting $\delta \to 0$ and using the monotone convergence theorem, we get
\[
u \le u_a^{**} + \int_{0}^{a_{0}}F^{-1}(\theta s)\di s \qquad \text{on } \,
\Omega \cap B_a(y)\backslash \{y\}.
\] 
The integral on the right hand side is finite. By continuity, the same inequality also holds at
the point $y$. Therefore, choosing $a_0$ sufficiently small, the estimate \eqref{eqsecondine} holds. This concludes the proof of the lemma.
\end{proof}

We are now ready to complete the proof of Theorem \ref{teoDir}(2). Recall that we are
dealing with a domain $\Omega$ with smooth boundary $\partial\Omega$, which at a
point $y$ has strictly negative mean curvature.
Fix $\varepsilon> 0$ and let $a_0 > 0$ small enough so that both Lemmas \ref{lema1} and
\ref{lema2} hold for $a< a_0$. This means that any positive solution
$u\in C^{\infty}(\Omega)\cap C(\overline{\Omega})$ of \eqref{SE} must satisfy the estimate
\begin{equation}\label{eq_u_phi}
u(y) \le \varepsilon + {\sup}_{\Omega \cap \partial B_a(y)} u \le 2 \varepsilon + {\sup}_{\partial\Omega \backslash B_a(y)} u - cf({\sup}_{\partial\Omega \backslash B_a(y) }u)
\end{equation}
for each $a< a_0$. On the other hand, choose an arbitrary positive constant $c_0>0$ and a positive boundary datum
$\phi \in C^\infty(\partial \Omega)$ satisfying
\[
\phi \equiv c_0 > 0 \quad \text{on } \, \partial\Omega \backslash B_a(y)
\quad\text{and}\quad \phi(y) > 2 \varepsilon + c_0 - cf(c_0).
\]
Then, from \eqref{eq_u_phi} it follows that $u(y)<\phi(y)$. Consequently, the Dirichlet
problem $\QQ[u]=0$ with prescribed $u\equiv\phi$ on $\partial\Omega$ does not admit any solution $u\in C^{\infty}(\Omega)\cap C(\overline{\Omega})$.

\section{Uniqueness of the grim-reaper cylinder}\label{sec5}

In this section, we prove Theorem \ref{teoGR}. We begin with the following observation:

\begin{lemma}\label{l1}
Let $M\subset \h^{n+1}$ be a properly immersed soliton with respect to $-\partial_0$ such that $\bi  M=\pi_1\cup\pi_2$, where $\pi_1$ and $\pi_2$ are parallel hyperplanes in $\bi \h^{n+1}$. Then $M$ is contained in the open region $\mathcal{U}$ bounded by the parallel hyperplanes $\Pi_1,\Pi_2\subset\h^{n+1}$ that meet $\bi \h^{n+1}$ orthogonally and satisfy $\bi \Pi_1=\pi_1$, $\bi \Pi_2=\pi_2$, and by the
half-cylinder $\mathcal{C}$ with $\bi \mathcal{C}=\pi_1\cup\pi_2$.
\end{lemma}

\begin{proof}
By Proposition \ref{prop_convexhull}, $M$ is contained in the slab between $\Pi_1$ and $\Pi_2$. 
Again from the strong maximum principle $M$ cannot touch $\Pi_1\cap\Pi_2$. Next, consider a 
spherical barrier $\SM$ centered at a point $q_\infty \in \bi \h^{m+1}$ equidistant, with respect to the 
Euclidean metric, from $\pi_1$ and $\pi_2$, and choose its radius to be small enough to satisfy
$\SM \cap M = \emptyset$. By increasing the radius, the strong maximum principle ensures that 
$M$ lies
above the half-sphere centered at $q_{\infty}$ and tangent to $\Pi_1\cup\Pi_2$. The conclusion 
follows since $\mathcal{C}$ is the envelope of such barriers for varying $q_{\infty}$.
\end{proof}

\subsection{Compactness and a maximum principle for varifolds}
Let us set recall some important
facts that we will need in the sequel.

\begin{definition}
Let $\{M_i\}_{i\in\natural{}}$ be a sequence of properly embedded hypersurfaces in a Riemannian manifold
$(N,\g)$.
We say that $\{M_i\}_{i\in\natural{}}$ has uniformly 
bounded area on compact subsets of $N$ if
$$
\limsup_{i\to\infty}|M_i\cap K|_{\g}<\infty
$$
for any compact subset $K$ of $N$.
\end{definition}

The following well-known theorem in geometric measure theory holds; see for example
\cite[Theorem 42.7]{simon}.

\begin{theorem}\label{WhiteCompactness}
	 Let $\{M_i
	\}$ be a sequence
	of minimal hypersurfaces in $ \real{n+1}$, with not necessary the canonical metric, whose area
	is locally bounded. Then a subsequence of $\{M_i\}$ converges weakly to a stationary integral
	varifold $M_\infty$.
\end{theorem}

Let us denote by
$$\mathcal{Z}=\big\{p\in \Omega: \limsup_{i\to\infty}|M_{i}\cap B_r(p)|_{\g}
= \infty \text{ for every } r>0\big\},$$
the set where the area blows up. Clearly $\mathcal{Z}$ is a closed set. It will be useful to have conditions that will imply that the set $\mathcal{Z}$ is empty. 
In this direction, White \cite[Theorems 2.6 and 7.3]{whi12} shows that under some 
natural conditions the set $\mathcal{Z}$ satisfies 
the same maximum principle as properly embedded minimal hypersurfaces without boundary. 

\begin{theorem}\label{thm:Controlling_area-blowup}
Let $(N,\g)$ be a smooth Riemannian $(n+1)$-manifold and $\{M_{i}\}_{i\in\natural{}}$ a sequence of properly
embedded minimal hypersurfaces without boundary in $(N,\g)$. Suppose that the area blow up set $\mathcal{Z}$ of $\{M_{i}\}_{i\in\natural{}}$ is contained in a closed $(n+1)$-dimensional region
$P \subset N$ with smooth, connected boundary $\partial P$ such that
$\g\big({\bf H},\xi\big)\ge 0,$
at every point of $\partial P$, where ${\bf H}$ is the mean curvature vector of
$\partial P$ and $\xi$ is the unit normal to the hypersurface $\partial P$ that points into 
$P$. If the set $\mathcal{Z}$ contains
any point of $\partial P$, then it contains all of $\partial P$.
\end{theorem}

\begin{remark}
The above theorem is a sub-case of a more general result. In fact the strong barrier
principle holds for sequences of properly embedded hypersurfaces (possibly with boundaries)
of Riemannian manifolds which are not necessarily minimal but they have uniformly bounded mean curvature. For more details we refer to \cite{whi12}.
\end{remark}

In the proof of Theorem F we will need the following strong maximum principle which is due to
Solomon and White \cite{solomon}.

\begin{theorem}\label{solomonmp}
Let $(N^{n+1},\g)$ be a Riemannian manifold with connected nonempty boundary $\partial N$ and that $N^{n+1}$ is mean convex, that is, that $\g({\bf H},\xi)\ge 0$ on $\partial N$ where ${\bf H}$ is the mean curvature vector of $\partial N$ and where $\xi$ is the unit
inward pointing normal of $\partial N$. Let $V$ be an $n$-dimensional stationary varifold
in $N^{n+1}$. If ${\rm spt}V$ contains a point of $\partial N$, then it
must contain all $\partial N$.
\end{theorem}

\subsection{The hyperbolic dynamic lemma}
Let $M\subset\h^{n+1}$ be a conformal soliton with respect to $-\partial_0$ satisfying the GR-property.
Without loss of generality we can choose coordinates so that $\pi_1$ and $\pi_2$ are respectively given
by the equations $\{x_0=0,x_1=a\}$ and $\{x_0=0,x_1=-a\}$, where $a$ is a positive constant.
Recall that the soliton property is preserved if we act on $M$ via isometries of the hyperbolic space
which fix the vector $-\partial_0$. Therefore, if $v=(0,0,v_2,\dots,v_m)$ is a vector
of $\h^{n+1}$ then the hypersurface
$$
M+v=(x_0,x_1,x_2+v_2,\dots,x_m+v_m)
$$
is again a soliton in $\h^{n+1}$ satisfying the GR-property.

\begin{lemma}\label{hyplem} 
Let $M\subset\h^{n+1}$ be a properly embedded soliton with respect to $-\partial_0$ satisfying
the GR-property. Suppose that $\{v_i\}_{i\in\n} \subset
\operatorname{span}\{ \partial_2,\dots,\partial_m \}$ is
a sequence of vectors and let $M_i= M+v_i$. Then, after passing to a subsequence, $\{M_i\}_{i\in\n}$
weakly converges to a connected stationary integral varifold $M_{\infty}$ with $\partial_{\infty}M_{\infty}=\partial_{\infty}M$.
\end{lemma}

\begin{proof}
First, by Lemma \ref{l1} we know that $M$ lies in the region $\mathcal{U}$. Let $\tau>0$ be the constant guaranteed by the GR-property, set
$$
\mathcal{U}_{\tau}=\mathcal{U}\cap \{\tau<x_0<{\sup}_{M}{x_0}\},
$$
and denote with $\mathcal{Z}$ the blow-up set of $\{M_i\}$. We split the proof in four steps. 

{\bf Step 1:} We show that the sequence $\{M_i\}_{i\in\n}$ has locally bounded
area with respect to the Ilmanen metric outside $\overline{\mathcal{U}}_{\tau}$. Due to the GR-property,
$$M\backslash\overline{\mathcal{U}}_{\tau}=\big(\mathcal{W}_1\cup\mathcal{W}_2\big)\cap\{x_0<\tau\},$$
where the wing $\mathcal{W}_j$ is the image of the graph of the function $\varphi_j:\mathcal{H}_j^{\tau}\to\R$, $j\in\{1,2\}$.
To simplify notation, for fixed $j$ we let
$$
\mathcal{W}=\mathcal{W}_{j},\,\,\, \varphi=\varphi_j\quad\text{and}\quad
\mathcal{H}=\mathcal{H}_j.
$$
Let $\nu$ be a Euclidean unit normal to
$\mathcal{H}$. Then $\nu= \cos \theta \partial_1 + \sin \theta \partial_0,$
where $\theta\in(-\pi/2,\pi/2)$. Notice that since $\mathcal{W}\subset\mathcal{U}$ there exists $C>0$
such that 
\begin{equation}\label{Cbox}
|\varphi(p)|<C\quad\text{for each}\quad p\in \mathcal{H}^{\tau}.
\end{equation}
Let us introduce the coordinates
$$
y_0 = \cos \theta x_1 +  \sin \theta x_0, \quad  y_1 = -\sin \theta x_1 + \cos \theta x_0, \quad y_k = x_k \,\, \text{for}\,\, k \ge 2.
$$
Thus $y=(y_1,\dots,y_m)$ are coordinates on $\mathcal{H}$.
Choose an Euclidean ball $B\subset\mathcal{H}^\tau$, consider the box $Q=[-C,C]\times B$ in coordinates $(y_0;y)$ and let $K=Q\cap\overline{\mathcal{U}}$; see Figure \ref{Pic-13}. By construction, $K$ is compact subset with piecewise smooth boundary in $\h^{n+1}$, and $\mathcal{W}\cap K$ is the image of the graph of the function $\varphi$ over the entire $\overline{B}$. We claim that the $\gI(n)$-area of $\{M_i\}_{i\in\n}$ on $K$ is uniformly bounded. Since for varying $B,C$ the sets $K$ cover the region $\mathcal{U} \cap \{x_0 < \tau\}$, we deduce that the sequence $\{M_i\}$ has locally bounded area outside $\overline{\mathcal{U}}_\tau$. We use a calibration method.

\begin{figure}[h!]
	\centering
	\includegraphics[width=.75\textwidth]{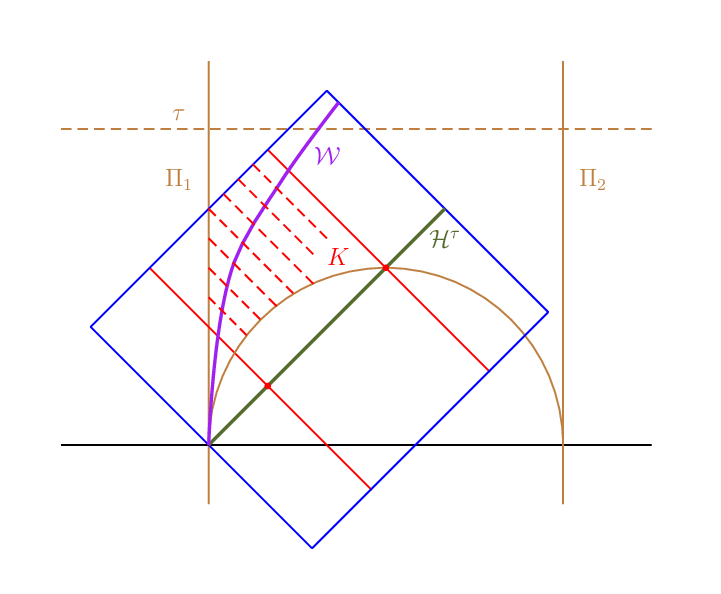}
	\caption{The subset $K$.}\label{Pic-13}
\end{figure}

Write for convenience
\[
\gI(n) = \lambda^2 \g_\R, \qquad \text{with } \quad 
\lambda(y_0;y)=
x^{-1}_0(y_0;y)e^{\frac{1}{nx_0(y_0;y)}},
\]
and notice that the unit normals $\xi$ and $\xi_{\Il}$ along $\mathcal{W}$ with respect to the Euclidean and
Ilmanen metric (say, pointing towards increasing $y_0$) are related by
$$
\xi_{\Il}(q)=\lambda^{-1}(q)\xi(q),\quad\text{for each}\quad q=(\varphi(y);y)\in\mathcal{W}.
$$
Extend $\xi$ on $K$ to be constant along the $y_0$-direction and accordingly extend $\xi_{\Il}$
on $K$ by
$$
\xi_{\Il}(y_0;y)=\lambda^{-1}(y_0;y)\xi(y_0;y), \quad\text{for each}\quad(y_0;y)\in K.
$$
We compute
\begin{equation}\label{eq_diver}
\begin{array}{lcl}
\diver_{\gI} \xi_{\Il} &= & \diver_{\g_\R}\xi_{\Il} + n \g_\R( D\log \lambda, \xi_{\Il}) \\[0.2cm]
& = & \lambda^{-1} \big( \diver_{\g_\R}\xi + (n-1) \g_\R( D\log \lambda, \xi)\big).
\end{array}
\end{equation}
Let $p = (y_0;y) \in K$, and let $q = (\varphi(y);y)$. Since $\mathcal{W}$ is $\gI(n)$-minimal and $|\xi_{\Il}|_{\gI}=1$  by construction on the entire $K$, we have
\[
\diver_{\gI} \xi_{\Il}(q) = 0, \qquad 
\diver_{\g_\R} \xi(p) = \diver_{\g_\R} \xi(q).
\] 
Evaluating \eqref{eq_diver} at $p$ and $q$ and using the last two identities, we get
\[
\diver_{\gI} \xi_{\Il}(p) = \frac{n-1}{\lambda(p)} \Big( \g_\R \big( D_p \log \lambda, \xi(p) \big)- \g_\R \big( D_q \log \lambda, \xi(q) \big)\Big).
\]
Since $\min_K x_0 > 0$, the function $\log \lambda$ is bounded in $C^1(K,\g_\R)$. So
there exists a constant $C_K$ such that
$$|\diver_{\gI} \xi_{\Il}| \le C_K\quad\text{on}\quad K.$$
Let $K'$ be the region of $K$ where $y_0 < \varphi(y)$. By the divergence theorem, 
\[
C_K |K|_{\gI}  \ge  \disp \int_{K'} \diver_{\gI} \xi_{\Il} \di x_{\gI}
 = |\mathcal{W} \cap K|_{\gI}
+ \int_{\partial K' \cap \partial K} \gI( \xi_{\Il}, \eta) \di \sigma_{\gI}
 \ge  |\mathcal{W} \cap K|_{\gI} - |\partial K|_{\gI}.
\]
Therefore,
$$|\mathcal{W} \cap K|_{\gI} \le |\partial K|_{\gI} + C_K |K|_{\gI}$$
is uniformly bounded, as claimed. 

{\bf Step 2:} From Step 1, $\mathcal{Z} \subset \overline{\mathcal{U}}_\tau$. Choose a spherical barrier
$\mathscr{B}$ not intersecting $\overline{\mathcal{U}}_\tau$, and increase its radius up to touching $\mathcal{Z}$. Theorem \ref{thm:Controlling_area-blowup} would imply that
$\mathscr{B}\subset\mathcal{Z}$, a contradiction. Thus, $\mathcal{Z} = \emptyset$. 

{\bf Step 3:}  From Steps 1 \& 2, the sequence $\{M_i\}$ has locally bounded $\gI(n)$-area. Hence, Theorem \ref{WhiteCompactness} guarantees the weak subconvergence of $\{M_i\}$ to a stationary integral varifold $M_\infty$. 
	
{\bf Step 4:} By the GR-property, $\bi M_i = \bi M$ for each $i\in\n$. Every point $p \in \bi \hmmaisum \backslash \bi M$ can be separated to $M$, 
hence to each $M_i$, by a small spherical barrier. Thus $\bi M_\infty \subset \bi M$. On the other hand, on each Euclidean ball 
$B$ centered at $p \in \bi M$ the GR-property and the almost monotonicity formula for $\gI(n)$-stationary varifolds guarantee a uniform lower bound 
for the $\gI(n)$-area of $M_i\cap B$. Therefore, $p$ belongs to ${\rm spt} M_\infty$.

This completes the proof.
\end{proof}

\subsection{Proof of Theorem \ref{teoGR}} 
By Lemma \ref{l1} and since $x_0$ is bounded on $M$, we have that $M\subset\mathcal{U}\cap\{x_0\le\sup x_0\}$.
Pick a grim-reaper cylinder $\mathscr{G}_h$ of height $h$ whose symmetry axis is parallel to the hyperplanes $\Pi_1$, $\Pi_2$,
and (Euclidean) equidistant to them. For $h$ small enough, $\mathscr{G}_h\cap M=\emptyset$. We claim that we can increase
$h$ up to a limit value $h^*$ in such a way that $\mathscr{G}_h\cap M=\emptyset$, for each $h<h^*$, and
$\bi\mathscr{G}_{h^*}=\pi_1\cup\pi_2$. Suppose that this is not the case. Then, necessarily,
$$
\begin{array}{ll}
(i)&\dist_{\gI}(\mathscr{G}_{h^*},M)=0,\\[0.2cm]
(ii)& \bi\mathscr{G}_{h^*}\,\,\text{ is contained in the open slab between $\pi_1$ and $\pi_2$.}
\end{array}
$$
By $(i)$ there exists a sequence $\{p_i\}_{i\in\n}\subset M$ such that $\dist(p_i,\mathscr{G}_{h^*})\to 0$.
Denote by $p^k_i$ the $x_k$-component of $p_i$ and define $v_i=(0,0,p_i^2,\dots,p_i^m)$.
Since $M$ is contained in $\mathcal{U}$ and because of $(ii)$,
it follows that there exists $\tau_0>0$ such that $p^0_i>\tau_0$ for each $i\in\n$. Therefore, up to a subsequence,
$p_i-v_i$ converges to a point $p_{\ell}\in\mathscr{G}_{h^*}$. Applying the
hyperbolic dynamic Lemma \ref{hyplem} to $M_i=M+v_i$, we obtain a limiting $\gI(n)$-stationary
varifold $M_{\infty}$ with $\bi M_{\infty}=\pi_1\cup\pi_2$. By construction $p_{\ell}\in M_{\infty}$ and
$M_{\infty}$ lies above $\mathcal{G}_{h^*}$. By Theorem \ref{solomonmp} we get that
$\mathscr{G}_{h^*} \subset M_{\infty}$,  contradicting condition $(ii)$.
To conclude, pick a large grim-reaper cylinder
$\mathscr{G}_{s}$ with the same axis as $\mathscr{G}_{h^*}$ and containing $\mathcal{U}\cap\{x_0\le\sup x_0\}$.
Decreasing $s$ and following the same argument as before, we can show that $\mathscr{G}_s\cap M=\emptyset$, for each $s>h^*$.
Hence $M\equiv \mathscr{G}_{h^*}$ and this completes the proof.


\begin{bibdiv}
\begin{biblist}

\bib{alias}{article}{
   author={Al\'ias, L.},
   author={De Lira, J-H.},
   author={Rigoli, M.}
   title={Mean curvature flow solitons in the presence of conformal vector fields},
   journal={J. Geom. Anal.}
   volume={30}
   date={2020},
   pages={1466-1529},
}

\bib{anderson1}{article}{
   author={Anderson, M.},
   title={Complete minimal hypersurfaces in hyperbolic $n$-manifolds},
   journal={Comment. Math. Helv.},
   volume={58},
   date={1983},
   pages={264-290},
}	

\bib{anderson2}{article}{
   author={Anderson, M.},
   title={Complete minimal varieties in hyperbolic space},
   journal={Invent. Math.},
   volume={69},
   date={1982},
   pages={477-494},
}

\bib{anderson3}{article}{
   author={Anderson, M.},
   title={The {Dirichlet} problem at infinity for manifolds of negative curvature},
   journal={J. Differ. Geom.},
   volume={18},
   date={1983},
   pages={701-722},
}


\bib{arezzo}{article}{
   author={Arezzo, C.},
   author={Sun, J.},
   title={Conformal solitons to the mean curvature flow and minimal submanifolds},
   journal={Math. Nachr.},
   volume={286}
   date={2013},
   pages={772-790},
}


\bib{baker}{book}{
   author={Baker, C.},
   title={The mean curvature flow of submanifolds of high codimension},
   series={Ph.D. Thesis, Australian National University. ArXiv version at arXiv:1104.4409},
   date={2010},
}

\bib{bernstein}{article}{
   author={Bernstein, S.},
   title={Conditions n{\'e}cessaires et suffisantes pour la possibilit{\'e} du probl{\`e}me de \emph{Dirichlet}},
   journal={C. R. Acad. Sci., Paris},
   volume={150},
   date={1910},
   pages={514--515},
}


\bib{bianchini}{book}{
   author={Bianchini, B.},
   author={Mari, L.},
   author={Pucci, P.},
   author={Rigoli, M.},
   title={Geometric analysis of quasilinear inequalities on complete manifolds. Maximum and
   compact support principles and detours on manifolds},
   series={Front. Math.},
   publisher={Cham: Birkh{\"a}user},
   date={2021},
}


\bib{bonorino}{article}{
   author={Bonorino, L.},
   author={Cast\'eras, J.-B.},
   author={Klaser, P.},
   author={Ripoll, J.},
   author={Telichevesky, M.},
   title={On the asymptotic Dirichlet problem for a class of mean curvature
   type partial differential equations},
   journal={Calc. Var. Partial Differential Equations},
   volume={59},
   date={2020},
   pages={Article No: 135},
}


\bib{casteras2}{article}{
	author = {Cast\'eras, J.-B.},
	author ={Heinonen, E.},
	author={ Holopainen, I.},
	year = {2017},
	pages = {1106-1130},
	title = {Solvability of minimal graph equation under pointwise pinching condition for sectional curvatures},
	volume = {27},
	journal = {J. Geom. Anal.},
}

\bib{casteras3}{article}{
	author = {Cast\'eras, J.-B.},
	author ={Heinonen, E.},
	author={ Holopainen, I.},
	year = {2019},
	pages = {917-950},
	title = {Dirichlet problem for $f$-minimal graphs},
	volume = {138},
	journal = {J. Anal. Math.},
	issue = {2},
}


\bib{casteras}{article}{
	author = {Cast\'eras, J.-B.},
	author={ Holopainen, I.},
	author ={Ripoll, J.},
	year = {2018},
	pages = {221-250},
	title = {Convexity at infinity in Cartan-Hadamard manifolds and applications to the asymptotic Dirichlet and 
	Plateau problems},
	volume = {290},
	journal = {Math. Z.},
}

\bib{clutterbuck}{article}{
   author={Clutterbuck, J.},
   author={Schn\"{u}rer, O.},
   author={Schulze, F.},
   title={Stability of translating solutions to mean curvature flow},
   journal={Calc. Var. Partial Differential Equations},
   volume={29},
   date={2007},
   pages={281-293},
}

\bib{CZ}{article}{
	author={Cheng, X.},
	author={Zhou, D.}, 
	title={Spectral properties and rigidity for self-expanding solutions of the mean curvature flows},
	journal={Math. Ann.},
	number={371}, 
	year={2018}, 
	issue={1-2}, 
	pages={371-389},
}

\bib{colding1}{article}{
   author={Colding, T.H.},
   author={Minicozzi II, W.P.},
   title={Generic mean curvature flow I; generic singularities},
   journal={Ann. of Math.},
   volume={175},
   date={2012},
   pages={755-833},
   }
   
   \bib{colding2}{article}{
   author={Colding, T.H.},
   author={Minicozzi II, W.P.},
   title={Smooth compactness of self-shrinkers},
   journal={Comment. Math. Helv.},
   volume={87},
   date={2012},
   pages={463-475},
   }
   
   
   \bib{CMR}{article}{
   author={Colombo, G.},
   author={Mari, L.},
   author={Rigoli, M.}
   title={Remarks on mean curvature flow solitons in warped products},
   journal={Discrete Contin. Dyn. Syst. Ser. S},
   volume={13},
   date={2020},
   issue={7},
   pages={1957-1991 (erratum on: Discrete Contin. Dyn. Syst. Ser. S16(2023), no.1, 187-196)},
}
  
   \bib{cil}{article}{
   author={Crandall, M.G.},
   author={Ishii, H.},
   author={Lions, P.-L.},
   title={User's guide to viscosity solutions of second order partial differential equations},
   journal={Bull. Am. Math. Soc., New Ser.},
   volume={27},
   date={1992},
   pages={1-67},
  }

\bib{guan}{article}{
	author={Guan, B.},
	author={Spruck, J.},
	title={Hypersurfaces of constant mean curvatures in hyperbolic space with prescribed asymptotic boundary at infinity},
	journal={Am. J. Math.}, 
	number={122},
	year={2000}, 
	pages={1039-1060},
}


\bib{lellis}{article}{
   author={De Lellis, C.},
   title={The size of the singular set of area-minimizing currents},
   journal={Advances in geometry and mathematical physics. Lectures given at the 
   geometry and topology conference at Harvard University, Cambridge, MA, USA, 
   2014},
   date={2016},
   pages={1-83},
}

\bib{lin}{article}{
	author={Lin, F.},
	title={On the Dirichlet problem for minimal graphs in hyperbolic space},
	journal={Invent. Math.}, 
	number={96},
	year={1989}, 
	pages={593-612},
}




\bib{dajczer2}{article}{
   author={Dajczer, M.},
   author={Hinojosa, P.-A.},
   author={De Lira, J.-H.},
   title={Killing graphs with prescribed mean curvature},
   journal={Calc. Var. Partial Differ. Equ.},
   volume={33},
   date={2008},
   pages={231-248},
}

\bib{dajczer3}{article}{
   author={Dajczer, M.},
   author={Ripoll, J.},
   title={An extension of a theorem of Serrin to graphs in warped products},
   journal={J. Geom. Anal.},
   volume={15},
   date={2005},
   pages={193-205},
}


\bib{eberlein}{article}{
	author={Eberlein, P.}, 
	author={O'Neill, B.},
	title={Visibility manifolds},
	journal={Pac. J. Math.},
	number={46},
	date={1973},
	pages={45-109},
}	

\bib{eschenburg}{article}{
   author={Eschenburg, J.-H.},
   title={Maximum principle for hypersurfaces},
   journal={Manuscripta Math.},
   volume={64},
   date={1989},
   pages={55-75},
}



\bib{gamamartin}{article}{
   author={Gama, E.-S.},
   author={Mart\'{\i}n, F.},
   title={Translating solitons of the mean curvature flow asymptotic to
   hyperplanes in $\Bbb{R}^{n+1}$},
   journal={Int. Math. Res. Not.},
   date={2020},
   volume={24},
   pages={10114-10153},
}

\bib{gilbarg}{book}{
   author={Gilbarg, D.},
   author={Trudinger, N.},
   title={Elliptic partial differential equations of second order},
   series={Grundlehren der Mathematischen Wissenschaften},
   volume={224},
   edition={2},
   publisher={Springer-Verlag, Berlin},
   date={1983},
}

\bib{gromov}{article}{
	author={Gromov, M.},
	title={Isoperimetric of waists and concentration of maps},
	journal={Geom. Funct. Anal.},
	volume={13},
	date={2003},
	pages={178-205},
}




%
%
%

\bib{huiskenpolden}{incollection}{
    author = {Huisken, G.},
    author = {Polden, A.},
    title = {Geometric evolution equations for hypersurfaces},
    booktitle = {Calculus of variations and geometric evolution problems ({C}etraro, 1996)},
    series = {Lecture Notes in Math.},
    volume = {1713},
    pages = {45-84},
    publisher = {Springer, Berlin},
    year = {1999},
}


\bib{hunger3}{article}{
	author={Hungerb\"uhler, N.},
	author={Mettler, T.},
	title={Soliton solutions of the mean curvature flow and minimal hypersurfaces},
	journal={Proc. Am. Math. Soc.},
	volume={140},
	date={2012},
	pages={2117-2126},
}

\bib{hunger1}{article}{
	author={Hungerb\"uhler, N.},
	author={Roost, B.},
	title={Mean curvature flow solitons},
	journal={Analytic aspects of problems in Riemannian geometry: elliptic PDEs, solitons and computer imaging. Selected papers based on the presentations of the conference, Landea, France},
	volume={13},
	date={2011},
	pages={129-158},
}

\bib{hunger2}{article}{
	author={Hungerb\"uhler, N.},
	author={Smoczyk, K.},
	title={Soliton solutions for the mean curvature flow},
	journal={Differ. Integral Equ.},
	volume={13},
	date={2000},
	pages={1321-1345},
}

\bib{ilmanen}{article}{
   author={Ilmanen, T.},
   title={Elliptic regularization and partial regularity for motion by mean
   curvature},
   journal={Mem. Amer. Math. Soc.},
   volume={108},
   date={1994},
   number={520},
   pages={x+90},
}

\bib{IRS}{article}{
	author={Impera, D.},
	author={Rimoldi, M.}, 
	author={Savo, A.},
	title={Index and first Betti number of $f$-minimal hypersurfaces and self-shrinkers},
	journal={Rev. Mat. Iberoam.},
	number={36},
	year={2020}, 
	issue={3}, 
	pages={817-840},
}

\bib{IR}{article}{
	author={Impera, D.},
	author={Rimoldi, M.}, 
	title={Rigidity results and topology at infinity of translating solitons of the mean curvature flow},
	journal={Commun. Contemp. Math.},
	number={19},
	year={2017}, 
	issue={6}, 
	pages={21 pp.},
}


\bib{jenkins}{article}{
	author={Jenkins, H.}, 
	author={Serrin, J.},
	title={The Dirichlet problem for the minimal surface equation in higher dimensions}, 
	journal={J. Reine Angew. Math.},
	number={229},
	year={1968}, 
	pages={170-187},
}

\bib{jorge}{article}{
   author={Jorge, L.},
   author={Tomi, F.},
   title={The barrier principle for minimal submanifolds of arbitrary codimension},
   journal={Ann. Global Anal. Geom.},
   volume={24},
   date={2003},
   pages={261-267},
}


\bib{Lady}{article}{
   author={Ladyzhenskaya, O.A.},
   author={Ural'tseva, N.N.},
   title={On H\"older continuity of solutions and their derivatives of linear and quasilinear elliptic and
   parabolic equations},
   journal={Transl., Ser. 2, Am. Math. Soc.},
   volume={61},
   date={1967},
   pages={207-269},
}

\bib{lang}{article}{
   author={Lang, U.},
   title={The existence of complete minimizing hypersurfaces in hyperbolic manifolds},
   journal={Int. J. Math.},
   volume={6},
   issue={1},
   date={1995},
   pages={45-58},
}


\bib{Lott}{article}{
   author={Lott, J.},
   title={Mean curvature flow in a Ricci flow background},
   journal={Commun. Math. Phys.},
   volume={313},
   date={2012},
   pages={517-533},
}

\bib{MMT}{article}{
	author={Magni, A.},
	author={Mantegazza, C.},
	author={Tsatis, E.},
	title={Flow by mean curvature inside a moving ambient space},
	journal={J. Evol. Equ.},
	number={13}, 
	year={2013}, 
	issue={3}, 
	pages={561-576},
}



\bib{fra15}{article}{
	author={Mart\'in, F.},
	author={P\'erez Garc\'ia, J.},
	author={Savas-Halilaj, A.},
	author={Smoczyk, K.},
	title={ A characterization of the grim reaper cylinder},
	journal={J. Reine Angew. Math.},
	volume={746},
	date={2019},
	pages={1-28},
}

\bib{fra14}{article}{
	author={Mart\'in, F.},
	author={Savas-Halilaj, A.},
	author={Smoczyk, K.},
	title={On the topology of translating solitons of the mean curvature flow},
	journal={Calc. Var. Partial Differential Equations},
	volume={54},
	date={2015},
	pages={2853-2882},
}


\bib{montiel}{article}{
	author={Montiel, S.},
	title={Unicity of constant mean curvature hypersurface in some Riemannian manifolds},
	journal={Indiana Univ. Math. J.},
	volume={48},
	date={1999},
	pages={711-748},
}

\bib{pucci}{book}{
   author={Pucci, P.},
   author={Serrin, J.},
   title={The maximum principle},
   series={Progress in Nonlinear Differential Equations and their
   Applications},
   volume={73},
   publisher={Birkh\"{a}user Verlag, Basel},
   date={2007},
}

\bib{ripoll}{article}{
	author={Ripoll, J.}, 
	author={Telichevesky, M.},
	title={Regularity at infinity of Hadamard manifolds with respect to some elliptic operators and applications    
	to asymptotic Dirichlet problems},
 	journal={Trans. Am. Math. Soc.}, 
 	number={367},
 	issue={3},
 	date={2015}, 
 	pages={1523-1541},
}

\bib{serrin2}{article}{
   author={Serrin, J.},
   title={The problem of Dirichlet for quasilinear elliptic differential
   equations with many independent variables},
   journal={Philos. Trans. Roy. Soc. London Ser. A},
   volume={264},
   date={1969},
   pages={413-496},
}



\bib{simon2}{article}{
   author={Simon, L.},
   title={Global estimates of H{\"o}lder continuity for a class of divergence-form elliptic equations},
   journal={Arch. Ration. Mech. Anal.},
   volume={56},
   date={1974},
   pages={253-272},
}

\bib{simon}{book}{
   author={Simon, L.},
   title={Lectures on geometric measure theory},
   volume={3},
   year={1983},
   series = {Proc. Cent. Math. Anal. Aust. Natl. Univ.},
   publisher={Australian National University, Centre for Mathematical Analysis, Canberra}
}




\bib{smoczyk}{article}{
	author={Smoczyk, K.},
	title={Mean curvature flow in higher codimension - Introduction and survey},
	journal={Global Differential Geometry,  Springer Proceedings in Mathematics},
	volume={12},
	year={2012},
	pages={231-274},
}

\bib{smoczyk1}{article}{
   author={Smoczyk, K.},
   title={A relation between mean curvature flow solitons and minimal submanifolds},
   journal={Math. Nachr.},
   volume={229},
   date={2001},
   pages={175-186},
}


\bib{solomon}{article}{
   author={Solomon, B.},
   author={White, B.},
   title={A strong maximum principle for varifolds that are stationary with respect to even parametric
   elliptic functionals},
   journal={Indiana Univ. Math. J.},
   volume={38},
   date={1989},
   pages={683-691},
}


\bib{Teli}{article}{
   author={Telichevesky, M.},
   title={A note on minimal graphs over certain unbounded domains of {Hadamard} manifolds},
   journal={Pac. J. Math.},
   volume={281},
   date={2016},
   pages={243-255},
}


\bib{whi12}{article}{
   author={White, B.},
   title={Controlling area blow-up in minimal or bounded mean curvature
   varieties},
   journal={J. Differential Geom.},
   volume={102},
   date={2016},
   pages={501-535},
}
   
\bib{whi10}{article}{
   author={White, B.},
   title={The maximum principle for minimal varieties of arbitrary codimension},
   journal={Comm. Anal. Geom.},
   volume={18},
   date={2010},
   pages={421-432},
}
   

\bib{yamamoto}{article}{
	author={Yamamoto, H.},
	title={Ricci-mean curvature flows in gradient shrinking Ricci solitons},
	journal={Asian J. Math.},
	volume={24},
	date={2020},
	pages={77-94},
}


\end{biblist}
\end{bibdiv}
\end{document}